\documentclass{article}

\usepackage{authblk}

\usepackage{amsmath, amssymb, amsthm}
\usepackage{graphicx,url,enumerate}

\newtheorem{theorem}{Theorem}[section]
\newtheorem{lemma}[theorem]{Lemma}
\newtheorem{proposition}[theorem]{Proposition}
\newtheorem{corollary}[theorem]{Corollary}
\theoremstyle{definition}
\newtheorem{definition}[theorem]{Definition}
\newtheorem{example}[theorem]{Example}
\newtheorem{remark}[theorem]{Remark}

\newcommand{\zco}{\mathfrak{C}\mspace{1mu}}

\newcommand{\zdu}{\mathfrak{D}\mspace{1mu}}

\newcommand{\CZ}{\mathbb{C}\mathfrak{Z}}
\newcommand{\Mat}{{\rm Mat}}
\newcommand{\Z}{\mathfrak{Z}}
\newcommand{\q}{\mathbf{q}}

\newcommand{\Ann}{{\rm Ann}_{\Z}}

\DeclareMathOperator{\sgn}{sgn}
\DeclareMathOperator{\spn}{span}

\newcommand{\zm}{\mathbf{0}}
\newcommand{\mI}{\mathbb{I}}

\numberwithin{equation}{section}

\begin{document}

\title{Spectral Properties of the Zeon Combinatorial Laplacian}

\author{G.~Stacey Staples\footnote{Email: sstaple@siue.edu}}
\affil{Department of Mathematics \&
Statistics\\
Southern Illinois University Edwardsville\\
Edwardsville, IL 62026-1653\\
USA}

\date{}  

\maketitle

\begin{abstract}
Given a finite simple graph $G$ on $m$ vertices, the zeon combinatorial Laplacian $\Lambda$ of $G$ is an $m\times m$ graph having entries in the complex zeon algebra $\CZ$.   It is shown here that if the graph has a unique vertex $v$ of degree $k$, then the Laplacian has a unique zeon eigenvalue $\lambda$ whose scalar part is $k$.  Moreover, the canonical expansion of the nilpotent (dual) part of $\lambda$ counts the cycles based at vertex $v$ in $G$.   With an appropriate generalization of the zeon combinatorial Laplacian of $G$, all cycles in $G$ are counted by  $\Lambda$.  Moreover when a generalized zeon combinatorial Laplacian $\Lambda$ can be viewed as a self-adjoint operator on the $\CZ$-module of $m$-tuples of zeon elements, it can be interpreted as a quantum random variable whose values reveal the cycle structure of the underlying graph.
\end{abstract}

\section{Introduction}

The {\em $n$-particle (real) zeon algebra} is a commutative $\mathbb{R}$-algebra generated by a fixed collection  $\{\zeta_{\{i\}}:1\le i\le n\}$ and scalar identity $1=\zeta_\varnothing$, whose generators (zeons) satisfy the zeon commutation relations \begin{equation*}
\zeta_{\{i\}}\zeta_{\{j\}}+\zeta_{\{j\}}\zeta_{\{i\}}=\begin{cases}
2\zeta_{\{i\}}\zeta_{\{j\}}\ne 0& i\ne j,\\
0&\text{\rm otherwise.}
\end{cases}
\end{equation*}   
We denote this algebra by $\Z_n$, although it has been denoted ${\mathcal{C}\ell_n}^{\rm nil}$ in a number of earlier papers where it was defined as a commutative subalgebra of a Clifford algebra of appropriate signature.  

Like fermions, the generators square to zero; like bosons, the generators commute.  Hence the name ``zeon algebra'', first suggested by Feinsilver~\cite{Feinsilverzeons}.

Combinatorial properties of zeons have proven useful in problems ranging from enumerating paths and cycles in finite graphs to routing problems in communication networks.  Many such enumeration problems are NP complete~\cite{Kar72}.  Where classical approaches to routing problems require construction of trees and the use of heuristics to prevent combinatorial explosion, the zeon algebraic approach avoids tree constructions and heuristics.  

Using zeon generators to define the nilpotent adjacency matrix of a finite graph allows one to count paths and cycles by considering powers of the matrix~\cite{ICCA7}.  Replacing the ordinary adjacency matrix with the nilpotent adjacency of a graph allows us to define the graph's zeon combinatorial Laplacian.  When $\Lambda$ denotes the zeon combinatorial Laplacian of a simple graph $G$ on $m$ non-isolated vertices, the paths between distinct vertices $v_i$ to $v_j$ in $G$ are enumerated by the off-diagonal entries of the inverse matrix $\Lambda^{-1}$~\cite{zeon laplacian}.  

The current paper is an extension of \cite{zeonspectraltheorem}, in which eigenvalues and eigenvectors of matrices with zeon entries were first studied.  The rest of the paper is laid out as follows.

Essential terminology and notation is presented in Section \ref{preliminaries}.  Matrix representations are briefly discussed in Section \ref{zeons from Pauli matrices}.  Zeros of zeon polynomials and the Fundamental Theorem of Zeon Algebra are reviewed in Section \ref{zeon zeros}.

Zeon matrices are viewed as $\CZ$-linear transformations on the zeon module $\CZ^m$ in Section \ref{matrices as transformations}, where essential properties are laid out.  In particular, the determinant, inverses, Gaussian elimination, inner products, the characteristic polynomial, eigenvalues, and eigenvectors are discussed.   

Main results concerning the zeon combinatorial Laplacian are developed in Section \ref{zeon Laplacian}.  In particular, we show that when a simple finite graph has a vertex $v$ of unique degree, the zeon combinatorial Laplacian has a unique eigenvalue $\lambda$ whose scalar part is equal to the degree of said vertex.  Further, the nilpotent (``dual'') part of $\lambda$ enumerates the vertex subsets on which cycles based at $v$ exist.  Moreover, the coefficients of basis blades in the canonical expansion of $\lambda$ count the cycles that exist on each vertex subset.  Associated with each zeon eigenvalue is a zeon eigenvector that we characterize in terms of the paths terminating at vertex $v$.   The eigenvalue $\lambda$ associated with a vertex of unique degree is characterized in Theorem \ref{laplacian eigenvalue}.  The eigenvector associated with $\lambda$ is characterized in Theorem \ref{Laplacian eigenvector}.

Generalizations of the zeon Laplacian appear in Section \ref{generalizations}.  In particular, a ``vertex-labeled zeon Laplacian'' is defined such that each vertex of the graph has an associated zeon eigenvalue whose dual part enumerates cycles based at the vertex.  Further, a quantum-probabilistic interpretation of the zeon Laplacian is obtained by defining a symmetric zeon Laplacian.  This ``quantum'' zeon combinatorial Laplacian can be viewed as a quantum observable whose values enumerate cycles contained in the associated graph. 

The paper concludes with parting comments and a discussion of avenues for further research.

\subsection{Preliminaries}\label{preliminaries}

For $n\in\mathbb{N}$,  let $\CZ_n$ denote the complex abelian algebra generated by the collection $\{\zeta_{\{i\}}: 1\le i\le n\}$ along with the scalar $1=\zeta_\varnothing$ subject to the following multiplication rules:
\begin{gather*}
\zeta_{\{i\}}\,\zeta_{\{j\}}=\zeta_{\{i,j\}}=\zeta_{\{j\}}\,\zeta_{\{i\}}\,\,\text{ \rm
for }i\ne j,\text{ \rm and}\\
{\zeta_{\{i\}}}^2=0\,\,\text{ \rm for }1\le i\le n.
\end{gather*}
It is evident that a general element $u\in\CZ_n$ can be expanded as $u=\displaystyle\sum_{I\in2^{[n]}}u_{I}\,\zeta_{I}$, or more simply as $\sum_I u_I\zeta_I$,  where $I\in2^{[n]}$ is a subset of the {\em $n$-set}, $[n]:=\{1,2,\ldots,n\}$, used as a multi-index, $u_{I}\in\mathbb{C}$, and $\displaystyle\zeta_{I}=\prod_{\iota\in I}\zeta_\iota$.  The algebra $\CZ_n$ is called the ($n$-particle) complex {\em zeon algebra}~\footnote{The $n$-particle (real) zeon algebra has been denoted by ${\mathcal{C}\ell_n}^{\rm nil}$ in a number of papers because it can be constructed as a subalgebra of the Clifford algebra $\mathcal{C}\ell_{2n,2n}$.}.  

As a vector space, this $2^n$-dimensional algebra has a canonical basis of  {\em basis blades} of the form $\{\zeta_I: I\subseteq [n]\}$.  The null-square property of the generators $\{\zeta_i:1\le i\le n\}$ guarantees that the product of two basis blades satisfies the following:\begin{equation}\label{blade product}
\zeta_I\zeta_J=\begin{cases}
\zeta_{I\cup J}& I\cap J=\varnothing,\\
0&\text{\rm otherwise.}
\end{cases}
\end{equation}
It should be clear that $\CZ_n$ is graded.  For non-negative integer $k$, the {\em grade-$k$ part} of element $u=\sum_I u_I\zeta_I$ is defined as \begin{equation*}
\langle u\rangle_k=\sum_{\{I:|I|=k\}}u_I\zeta_I.
\end{equation*}

Given a zeon element $u=\sum_I u_I\zeta_I$, the {\em complex conjugate} of $u$ is defined by $\overline{u}=\sum_I \overline{u_I}\zeta_I$.
The $n$-particle {\em real zeon algebra} $\Z_n$ is the subalgebra of $\CZ_n$ defined by $\Z_n=\{u\in\CZ_n: \overline{u}=u\}.$

An inner product is defined on $\CZ_n$ by sesquilinear extension of \begin{equation*}
\left< \sum_I u_I \zeta_I, \alpha\zeta_J\right> = \overline\alpha u_J,
\end{equation*}
so that any element $u\in\CZ_n$ has canonical expansion $u=\sum_I \langle u,\zeta_I\rangle\zeta_I.$

Let $u\in\CZ_n$ such that $\zco u\ne 0$ and let $\kappa$ denote the index of nilpotency~\footnote{The index of nilpotency of $\zdu u$ is the least positive integer $\kappa$ such that  $(\zdu u)^\kappa=0$.} of $\mathfrak{D}u$.  Then $u$ is invertible if and only if $\zco u\ne0$, and the unique inverse of $u$ is given by \begin{equation*}
u^{-1}=\frac{1}{\zco u}\sum_{j=0}^{\kappa-1}(-1)^{j}(\zco u)^{-j}(\mathfrak{D}u)^{j}.
\end{equation*}

The invertible elements of $\CZ_n$ form a multiplicative Abelian group, denoted ${\CZ_n}^\times$, and the nilpotent elements form a maximal ideal, denoted ${\CZ_n}^\circ$.

Given $z\in\CZ_n$, we write $\zco z=\langle z\rangle_0$ for the {\em complex (scalar) part} of $z$, and $\mathfrak{D}z=z-\zco z$ for the {\em dual (nilpotent) part} of $z$.

\subsection{Matrix Representations}\label{zeons from Pauli matrices}

It is worth noting that computations in $\CZ_n$ can be performed with complex square matrices of order $2^n$, as described below.  The representation of an $m\times m$ matrix with $\CZ_n$ entries is then a square complex matrix of order $m2^n$.  

Begin by defining the matrix $\eta$ by 
\begin{equation}\label{eta def}
\eta=\begin{pmatrix}
0&1\\
0&0
\end{pmatrix}
\end{equation}
and observe that in terms of the Pauli matrices, $\eta=\frac{1}{2}(\sigma_x+\overline{i}\sigma_y)$. Let $\eta^\dag$ denote the Hermitian conjugate of $\eta$.  In terms of Pauli matrices, $\eta$ satisfies the following identities:
\begin{align*}
\eta+\eta^\dag &=\sigma_x,\\
\eta-\eta^\dag &=\overline{i}\sigma_y,\\
\eta\eta^\dag -\eta^\dag\eta&=\sigma_z,\\
\eta\eta^\dag+\eta^\dag\eta &=\sigma_0.
\end{align*}

In $\CZ_n$, the matrix representation of $\zeta_{\{j\}}$ is given by \begin{equation*}
\zeta_{\{i\}}\mapsto {\sigma_0}^{(i-1)}\otimes\eta\otimes{\sigma_0}^{\otimes(n-i)}
\end{equation*}
for each $i=1,\ldots, n$.

\begin{example}
In $\CZ_1$, the element $a+b\,\zeta_{\{1\}}$ is represented by \begin{equation*}
a\,\sigma_0+b\,\eta= a\begin{pmatrix}
1&0\\
0&1\end{pmatrix}+b\begin{pmatrix}0&1\\0&0\end{pmatrix}= \begin{pmatrix}
a&b\\
0&a\end{pmatrix},
\end{equation*}
while in $\CZ_2$, the element $a+b\,\zeta_{\{1\}}+c\,\zeta_{\{2\}}+d\,\zeta_{\{1,2\}}$ is represented by \begin{align*}
a \,{\sigma_0}^{\otimes 2}+b\,\eta\otimes{\sigma_0}+c\,\sigma_0\otimes\eta+d\,\eta^{\otimes 2}&=a\begin{pmatrix}
1&0&0&0\\
0&1&0&0\\
0&0&1&0\\
0&0&0&1\end{pmatrix}+b\begin{pmatrix}
0&0&1&0\\
0&0&0&1\\
0&0&0&0\\
0&0&0&0\end{pmatrix}\\
&+c\begin{pmatrix}
0&1&0&0\\
0&0&0&0\\
0&0&0&1\\
0&0&0&0\end{pmatrix}+d\begin{pmatrix}
0&0&0&1\\
0&0&0&0\\
0&0&0&0\\
0&0&0&0\end{pmatrix}\\
&=\begin{pmatrix}
a&c&b&d\\
0&a&0&b\\
0&0&a&c\\
0&0&0&a\end{pmatrix}.
\end{align*}
\end{example}

Throughout the current work, computations involving  $m\times m$ zeon matrices are performed symbolically (with {\em Mathematica}) using $m\times m$ matrices and set operations on multi indices in computations.  

\subsection{Spectrally Simple Zeros of Complex Zeon Polynomials}\label{zeon zeros}

Before any discussion of eigenvalues is possible, zeros of zeon polynomials must be addressed.  

Given a complex zeon polynomial $\varphi(u)=\alpha_m u^m+\cdots+\alpha_1 u+\alpha_0$, where $\alpha_j\in\CZ_n$ for $j=0, \ldots, m$, a {\em complex polynomial} $f_\varphi:\mathbb{C}\to\mathbb{C}$ is induced by \begin{equation*}
f_\varphi(z)=\sum_{\ell=0}^m (\zco \alpha_\ell) z^\ell.
\end{equation*}
It follows that $f_\varphi(\zco u)=\sum_{\ell=0}^m (\zco \alpha_\ell) (\zco u)^\ell=\zco(\varphi(u))$ so that $f_\varphi\circ\zco=\zco\circ \varphi$.  

It is evident that if $\lambda\in\CZ_n$ is a zero of the zeon polynomial $\varphi$, then $\zco\lambda$ is a zero of the induced complex polynomial $f_\varphi$.  If $f_\varphi$ has no complex zeros, then $\varphi$ has no zeon zeros.  Further, if $f_\varphi$ has a multiple zero $w_0\in\mathbb{C}$, $\varphi$ may or may not have a zero $w$ satisfying $\zco w=w_0$.   For this reason, we focus on ``spectrally simple'' zeros of polynomials with zeon coefficients.  

Letting $\varphi(u)\in{\CZ_n}[u]$ be a nonconstant monic zeon polynomial, we consider $\lambda\in{\CZ_n}$ to be a {\em simple} zero of $\varphi$ if $\varphi(u)=(u-\lambda)g(u)$ for some zeon polynomial $g$ satisfying $g(\lambda)\ne 0$.  Recalling that the spectrum of an element $u$ in a unital algebra is the collection of scalars $\lambda$ for which $u-\lambda$ is not invertible, it follows that when $u\in\CZ_n$, the spectrum of $u$ is the singleton $\{\lambda=\zco u\}$.  

\begin{definition}
A simple zero $\lambda\in{\CZ_n}$ of $\varphi(u)$ is said to be a {\em spectrally simple} if $\zco \lambda$ is a simple zero of the complex polynomial $f_\varphi(z)$.
\end{definition}  

The Fundamental Theorem of Zeon Algebra below (established in \cite{simplezeonzeros}) guarantees the existence of a simple  zeon zero of $\varphi(u)$ when the complex polynomial $f_\varphi(z)$ has a simple complex zero.

\begin{description}
\item [{\bf Fundamental Theorem of Zeon Algebra.}]
Let $\varphi(u)$ be a monic polynomial of degree $m$ over $\CZ_n$, and let $f_\varphi(z)$ be the complex polynomial induced by $\varphi$.  If $\lambda_0\in\mathbb{C}$ is a simple zero of $f_\varphi(z)$, then $\varphi(u)$ has a simple zero $\lambda$ such that $\zco\lambda = \lambda_0$. Moreover, such a zero $\lambda$ is unique.
\end{description}
Clearly, the requirement that $\varphi(u)$ be monic can be relaxed by simply requiring the leading coefficient $\alpha_m$ of $\varphi$ to be invertible.  Further, if $f_\varphi(z)$ is a nonconstant polynomial whose zeros are all simple, then $\varphi(u)$ has exactly $m$ complex zeon zeros.  In this case, we can say $\varphi$ {\em spectrally splits} over $\CZ_n$.

\subsection{The Complex Zeon Algebra $\CZ$}

The results above are easily generalized to zeon algebras of arbitrary dimension.  To that end, we remove the restriction on the number of generators and let $\CZ$ denote the complex zeon algebra defined as follows.

\begin{definition}
Let $\CZ$ denote the infinite-dimensional complex Abelian algebra generated by the collection $\{\zeta_{\{i\}}: i\in \mathbb{N}\}$ along with the scalar $1=\zeta_\varnothing$ subject to the following multiplication rules:
\begin{gather*}
\zeta_{\{i\}}\,\zeta_{\{j\}}=\zeta_{\{i,j\}}=\zeta_{\{j\}}\,\zeta_{\{i\}}\,\,\text{ \rm
for }i\ne j,\text{ \rm and}\\
{\zeta_{\{i\}}}^2=0\,\,\text{ \rm for }i\in \mathbb{N}.
\end{gather*}
For each finite subset $I$ of $\mathbb{N}$, define $\displaystyle\zeta_{I}=\prod_{\iota\in I}\zeta_\iota$.  Letting the finite subsets of positive integers be denoted by $[\mathbb{N}]^{<\omega}$, the algebra $\CZ$ has a canonical basis of the form $\{\zeta_I: I\in [\mathbb{N}]^{<\omega}\}$.   Elements of this basis are referred to as the {\em basis blades} of $\CZ$.   The algebra $\CZ$ is called the  {\em (complex) zeon algebra}.
\end{definition}

Any element $u\in\CZ$ can be expressed as a linear combination of basis blades indexed by finite subsets of $\mathbb{N}$.  Hence, there exists $N\in\mathbb{N}$ such that $u\in \CZ_n$ for all $n\ge N$.   With this in mind, the notation is simplified throughout the rest of the paper by using $\CZ$ in place of $\CZ_n$.

\section{Zeon Matrices as $\CZ$-Linear Transformations}\label{matrices as transformations}
 
For $m>1$, the $m$-tuples of zeon elements, ${\CZ}^m$, constitute a $\CZ$-module.  The $m\times m$ matrices over $\CZ$, denoted by $\Mat(m,\CZ)$, represent the $\CZ$-linear transformations of ${\CZ}^m$.  

A matrix $A\in\Mat(m,\CZ)$ can naturally be written as a sum of the form $A=A_\varnothing+\mathfrak{A}$ of a complex-valued matrix $A_\varnothing = \zco A$ and a nilpotent zeon-valued matrix $\mathfrak{A}=\zdu A\in\Mat\CZ^\circ$. 

\subsubsection*{Properties of Nilpotent Zeon Matrices} 
The following properties can be verified easily.
\begin{enumerate}[i.]
\item The collection of $m\times m$ square matrices over ${\CZ}^\circ$ forms an ideal.     
\item Each matrix in $\Mat(m,\CZ^\circ)$ is nilpotent.
\item Any power series of a zeon matrix $A\in\Mat(m,\CZ)$ reduces to a finite sum, and is therefore convergent; i.e., given $A=\zco A+\zdu A\in\Mat(m,\CZ)$, the following holds:
\begin{equation*}
\sum_{\ell=0}^\infty \alpha_\ell A^\ell=\sum_{\ell=0}
^{\kappa(\zdu A)}\alpha_\ell A^\ell,
\end{equation*}
where $\kappa(\zdu A)$ is a finite nonnegative integer, and $\alpha_\ell\in\CZ$ for each $\ell$.
\end{enumerate}

\subsection{The Determinant}

The determinant of $A\in\Mat(m,\CZ)$ is defined in the usual way by \begin{equation*}
|A|=\sum_{\sigma\in \mathcal{S}_m}\sgn(\sigma)\prod_{i=1}^m a_{i\,\sigma(i)},
\end{equation*}
where $\mathcal{S}_m$ is the symmetric group of order $m!$ and $\sgn (\sigma)$ is the signature of the permutation $\sigma$.

\begin{lemma}[Properties of the determinant]
Let $A$ and $B$ be $m\times m$ matrices over $\CZ$, and let $\alpha\in\CZ$.  Then the following hold:\begin{align*}
|AB|&=|A| |B|,\\
|\alpha A|&=\alpha^m |A|.
\end{align*}
In particular, $|A^{-1}|=|A|^{-1}$ when $A$ is invertible.
\end{lemma}

\begin{description}
\item [{\bf Zeon matrix inverse.}] Let $A=(a_{ij})$ be a square matrix having entries from $\CZ$, and write $A=\zco A + \mathfrak{D}A$, where $\zco A=(\zco a_{ij})$.  It follows that $A$ is invertible if and only if $\zco A$ is invertible.  In this case, the inverse is given by \begin{equation*}
A^{-1}=(\zco A)^{-1}\sum_{\ell=0}^{\kappa(\mathfrak{D}A(\zco A)^{-1})}(-1)^\ell(\mathfrak{D}A(\zco A)^{-1})^{\ell}.
\end{equation*}
\end{description}

We note that a zeon matrix $A\in\Mat(m,\CZ)$ is singular if and only if its determinant is nilpotent.  Inverses can also be obtained via Gaussian elimination.

\subsubsection*{Zeon Gaussian Elimination}

The effects of elementary row operations on the determinant of an $m\times m$ zeon matrix are as follows:
\begin{enumerate}[i.]
\item $A\mapsto A'$ by exchanging two rows negates the determinant: $|A|=-|A'|$;
\item $A\mapsto A'$ by multiplying a row of $A$ by a zeon constant $u$ changes the determinant by the factor $u$: $|A'|= u|A|$;  
\item adding a zeon multiple of one row to another leaves the determinant unchanged.
\end{enumerate}

Elementary zeon matrices can be defined as matrices obtained from the identity matrix by elementary zeon row operations, and inverse matrices can be computed using Gaussian elimination.

\subsection{The Zeon Module Inner Product}

Let $x,y\in{\CZ}^m$.  Writing $x$ and $y$ as column matrices $y=(y_1, \ldots, y_m)^\intercal$ and \hfill\break $x=(x_1, \ldots, x_m)^\intercal$, the {\em zeon inner product} of $x$ and $y$ is defined by
\begin{equation}\label{inner prod def}
\langle x\vert y\rangle = y^\dag x,
\end{equation}
where $y^\dag=\overline{y}^\intercal$ denotes the complex conjugate transpose of $y$.

While $\langle x\vert x\rangle$ is generally not scalar-valued, the scalar part of $\langle x\vert x\rangle$  defines a seminorm on $\CZ$.   In particular, $\zco \langle x\vert x\rangle\ge 0$ for all $x\in{\CZ}^m$, so we define the {\em spectral seminorm} of $x$ by \begin{equation}
|x|_\star=\zco\langle x\vert x\rangle^{1/2}.
\end{equation}

A zeon vector $v\in{\CZ}^m$ is said to be {\em null} if its spectral seminorm is zero; i.e., $\zco\langle v, v\rangle = 0$, or equivalently, $\langle v,v\rangle\in{\CZ}^\circ$.  A non-null zeon vector $v$ is said to be {\em normalized} if and only if $\langle v,v\rangle =1$.  Any non-null $x\in{\CZ}^m$  can be normalized via the mapping $x\mapsto \hat{x}$ where \begin{equation}
\hat{x}=\langle x\vert x\rangle^{-1/2} x.
\end{equation}

\subsection{Eigenvectors, Eigenvalues, and the Characteristic Polynomial}

Let  $T\in{\rm End}({\CZ}^m)$.   If there exists an element $\lambda\in\CZ$ such that  the kernel of $T-\lambda\mI$ is a nonempty submodule of ${\CZ}^m$, then $\lambda$ is said to be an {\em eigenzeon} of $T$.  Any non-null $x\in \ker(T-\lambda\mI)$ is then said to be a {\em zeon eigenvector} of $T$ corresponding to {\em eigenzeon} $\lambda$.

Regarding endomorphisms of ${\CZ}^m$ as matrices, a non-null zeon vector $\xi\in {\CZ}^m$ is said to be an {\em  eigenvector} of $A\in \Mat(m,\CZ)$ if there exists an element $\lambda\in{\CZ}$ such that $A\xi=\lambda \xi$.  In this case, $\lambda$ is said to be an {\em eigenvalue} of $A$ associated with the eigenvector $\xi$.

Given a square matrix $A\in \Mat(m,\CZ)$,  the {\em characteristic polynomial} of $A$ is defined in the usual way as the zeon polynomial $\chi_A(t)=|t\mI-A|$.  When there exists $\lambda\in\CZ$ such that $\chi_A(\lambda)=0$, there necessarily exists a non-null zeon vector $\xi\in{\CZ}^m$ such that $A\xi =\lambda\xi$.

Some facts about zeon matrices follow.
\begin{enumerate}
\item Let $A\in \Mat(m,\CZ)$.  If $\chi_{\zco A}(u)$ has simple complex zero $u_0$, then $A$ has a unique spectrally simple zeon eigenvalue $\lambda$ satisfying $\zco \lambda = u_0$.

\item If $A$ is invertible, then all eigenvalues of $A$ are invertible.

\item Let $A\in\Mat(m,\CZ)$.  If $\lambda$ is a zeon eigenvalue of $A$ associated with eigenvector $\xi$,  then $\zco \lambda$ is an eigenvalue of $\zco A$ associated with eigenvector $\zco \xi$.

\item The mapping $A\mapsto \zco A$ is an algebra homomorphism  \break $\Mat(m,\CZ)\to {\rm End}(\mathbb{C}^m)$.  Consequently, for any positive integer $k$, $\zco(A^k)=(\zco A)^k$.   Moreover, $A\in\Mat(m,\CZ)$ is singular if and only if $\zco A$ is singular.

\item For all $A\in\Mat(m,\CZ)$,  $\zco |A| = |\zco A|$. 

\item A matrix $A \in\Mat(m,\CZ)$ is nilpotent if and only if $\zco A$ is nilpotent.

\end{enumerate}

\subsection*{Zeon Eigenvectors}

A few remarks concerning linear independence of vectors in ${\CZ}^m$ are in order.  A subset $\{v_1, \ldots, v_k\}$ of ${\CZ}^m$ is said to be {\em linearly independent} if and only if for all coefficients $\alpha_1, \ldots, \alpha_k\in\CZ$, 
\begin{equation*}
\alpha_1v_1+\cdots+\alpha_kv_k=\zm
\end{equation*} 
implies $\alpha_1=\alpha_2=\cdots=\alpha_k=0$.

Note that for any nonzero $v\in {\CZ}^m$, the singleton $\{v\}$ can only be linearly independent if it has at least one invertible component.  That is, $\{v\}$ linearly independent implies $v\notin ({\CZ}^\circ)^m$; equivalently, $\{v\}$ is linearly independent if and only if $\zeta_{[n]}v\ne \zm$.

\begin{proposition}
Let $A\in \Mat(m,\CZ)$.  Let $v_1,v_2\in{\CZ}^m$ be eigenvectors of $A$ associated with zeon eigenvalues $\lambda_1, \lambda_2$, respectively, where $\zco \lambda_1\ne \zco \lambda_2$.  Then, $v_1$ and $v_2$ are linearly independent.
\end{proposition}

\begin{proof}
Suppose, to the contrary, that for some nonzero constants $\alpha_1, \alpha_2\in\CZ$, the following holds:\begin{equation}\label{orth1}
\zm=\alpha_1v_1+\alpha_2v_2=(\alpha_1v_1+\alpha_2v_2)\lambda_1.
\end{equation}
Further, linearity of $A$ implies
\begin{equation}\label{orth2}
\zm=A(\alpha_1v_1+\alpha_2v_2)=\alpha_1\lambda_1v_2+\alpha_2\lambda_2v_2.
\end{equation}
Subtracting \eqref{orth2} from \eqref{orth1} gives $\alpha_2(\lambda_2-\lambda_1)v_2=\zm$.
Since, $\zco\lambda_1\ne\zco\lambda_2$, we conclude that $\lambda_2-\lambda_1$ is invertible.  Hence, $\alpha_2 v_2=\zm$, which further implies $\alpha_1v_1=\zm$ (by \eqref{orth1}), so that $v_1,v_2\in({\CZ}^\circ)^m$, a contradiction.   
\end{proof}

\begin{remark}
Note that the only zeon eigenvalue that can be associated with the nullspace of matrix $A\in\Mat(m,\CZ)$ is zero.  If $\lambda v=0$ for nonzero $\lambda$, then $\lambda$ is nilpotent and $v\in{{\CZ}^\circ}^m$, violating the requirement that an individual eigenvector be linearly independent.  
\end{remark}

\section{The Zeon Combinatorial Laplacian}\label{zeon Laplacian}

For purposes of the discussion at hand, a {\em graph} on $n$ vertices is an ordered pair $G=(V,E)$ of sets.  Elements of $V=\{v_1, \ldots, v_n\}$ are the {\em vertices} of $G$, and elements of $E\subseteq\{\{v_i, v_j\}: 1\le i<j\le n\}$ are the {\em edges} of $G$.  Two vertices $v_i,v_j\in V$ are {\em adjacent} if there exists an edge $\{v_i,v_j\}\in E$.  

A {\em $k$-walk} $(v_0,\ldots,v_k)$ in a graph $G$ is a sequence of vertices in $G$ with {\em initial vertex} $v_0$ and {\em terminal vertex} $v_k$ such that there exists an edge
$\{v_j,v_{j+1}\}\in E$ for each $0\le j\le k-1$.  A $k$-walk contains $k$ edges.  A {\em path} is a walk in which no vertex appears more than once; equivalently, a path is a {\em self-avoiding} walk.  A {\em closed $k$-walk} is a $k$-walk whose initial vertex is also its terminal vertex.  A {\em $k$-cycle} is a closed $k$-path with the exception $v_0=v_k$.  A {\em Hamiltonian cycle} is an $n$-cycle in a graph on $n$ vertices; i.e., it contains $V$. 

\begin{description}
\item[{\bf PWICs.}] For convenience, we further designate a {\em path with initial cycle} (PWIC)  from $v_0$ to $v_t$ (where $v_0$ and $v_t$ are distinct vertices of $G$) to be a walk containing exactly one cycle based at the initial vertex $v_0$.  In particular, PWICs are walks  in which no vertex is repeated except for the initial vertex $v_0$, which is repeated exactly once.
\end{description}

The significance of PWICs will become evident in the proof of Theorem \ref{nil-structure theorem}.  First, we recall the definition of the nilpotent adjacency matrix associated with a graph.  

\begin{definition}
Let $G=(V,E)$ be a simple graph on $n$ labeled vertices $V=\{v_1, \ldots, v_n\}$ and edges $E$ consisting of unordered pairs of vertices.  The {\em nilpotent adjacency matrix of $G$} is defined as the $n\times n$ matrix $\Psi=(\psi{ij})$ having entries in ${\CZ}$ determined as follows:
\begin{equation*}
\psi_{ij}=\begin{cases}
0&\{v_i, v_j\}\notin E,\\
\zeta_{\{j\}}& \{v_i, v_j\}\in E.
\end{cases}
\end{equation*}
\end{definition}

Recalling Dirac notation,  $\langle i \vert$ will denote the row vector whose entries are all zero, except the $i$th column, which is $1$.  Similarly, $\vert j\rangle$ will denote the column vector with $1$ in the $j$th row and zeros elsewhere.   Further, let us define $\langle \zeta_{\{i\}} \vert$ to denote the row vector whose entries are all zero, except the $i$th column, which is $\zeta_{\{i\}}$.  With this notation in mind, the following result is a refinement of one first established in \cite{ICCA7}.  

\begin{theorem}[Nil-Structure Theorem]\label{nil-structure theorem}
Let $\Psi$ be the nilpotent adjacency matrix of an
$n$-vertex graph $G$, and let $k\in\mathbb{N}$.  For $1\le i\ne j \le n$, 
\begin{equation*}
\left<\zeta_{\{i\}}\vert\Psi^k\vert j\right>=
\sum_{\genfrac{}{}{0pt}{}{I\subseteq V}{|I|=k+1}}\omega_I \zeta_{I},\end{equation*}
where $\omega_I$ denotes the number of $k$-paths from $v_i$ to $v_j$ on vertices indexed by $I$.  Further,  when $1\le i \le n$,
\begin{equation*}
\left<i\vert\Psi^k\vert i\right>=\sum_{\genfrac{}{}{0pt}{}{I\subseteq V}{|I|=k}}\omega_I \zeta_{I},\end{equation*}
where $\omega_I$ denotes the number of $k$-cycles on vertex set $I$ based at $v_i\in I$.  
\end{theorem}

\begin{proof}
Because the generators of $\CZ$ square to zero, a straightforward inductive argument shows that the nonzero terms of $\left<i\vert\Psi^k\vert j\right>$ are multivectors corresponding to two types of $k$-walks from $v_i$ to $v_j$: (i.) paths (i.e., walks with no repeated vertices) and (ii.) PWICs in which $v_i$ is repeated exactly once at some step but are otherwise self-avoiding.   For $i\ne j$, multiplication by $\zeta_{\{i\}}$ cancels walks of the second type, so that nonzero terms represent paths from $v_i$ to $v_j$.   

When $i=j$, walks of the second type are zeroed in the $k^{\rm th}$ step when the walk is closed. Hence, terms of $\left<i\vert \Psi^k\vert i\right>$ represent the collection of $k$-cycles based at $v_i$.
\end{proof}

\begin{remark}
At times, it will also be convenient to use the notation $\langle v \vert$ or $\vert v\rangle$ to denote the unique row (resp. column) unit vector associated with vertex $v$.   
\end{remark}

\begin{definition}[Zeon Laplacian]
Let $\Psi$ denote the nilpotent adjacency matrix of a simple graph $G$ on $m$ vertices, and let $\Delta={\rm diag}(d_1,\ldots, d_m)$ be the diagonal matrix of vertex degrees associated with $G$.  The {\em zeon Laplacian of $G$} is defined by \begin{equation}\label{zeon Laplacian def}
\Lambda=\Delta-\Psi.
\end{equation}
\end{definition}

It is worthwhile to note that $\zco(\Lambda)=\Delta$, which is invertible if and only if all of its diagonal entries are nonzero; hence, we have the following lemma.

\begin{lemma}
The zeon Laplacian of a simple graph $G$ is invertible if and only if $G$ has no isolated vertices.
\end{lemma}

Given a multi-index $I$, it is convenient to define \begin{equation*}
\deg(v_I)=\prod_{j\in I}\deg (v_j).
\end{equation*}

\subsection*{Row Sums of the Zeon Laplacian}

A feature of the ordinary combinatorial Laplacian is that all row sums are zero.  Such is not the case for the zeon combinatorial Laplacian.  However, vertex degree is directly related to the index of nilpotency of the dual part of the row sum.

Letting $\Lambda=(\lambda_{ij})\in\Mat(\CZ,m)$ be the zeon combinatorial Laplacian of a graph $G$ on $m$ vertices, the $i$th {\em row sum} $r_i$ of $\Lambda$ is 
\begin{equation*}
r_i=\sum_{j=1}^m \lambda_{ij}= \deg(v_i)-\sum_{\{j: \{v_j,v_i\}\in E\}}\zeta_{\{j\}},
\end{equation*}
so that $\zco r_i=\deg(v_i)$ and $\zdu r_i=-\sum_{\{j: v_j\in\mathcal{N}(v_i)\}}\zeta_{\{j\}}$.  It follows that \begin{equation}\label{row sum dual part}
(\zdu r_i)^{\deg(v_i)}=|\mathcal{N}(v_i)|!\zeta_{I},
\end{equation}
where $I=\{\ell: v_\ell\in\mathcal{N}(v_i)\}$; i.e., $I$ is the set of indices of neighbors of $v_i$ in the graph $G$.  Further, the index of nilpotency of $\zdu(r_i)$ is $\kappa(\zdu(r_i))=\deg(v_i)+1$.   

Hence, 
\begin{equation}\label{row sum inverse}
{r_i}^{-1}=\frac{1}{\deg(v_i)}\sum_{k=0}^{\deg(v_i)}\left(\frac{-\zdu(r_i)}{\deg(v_i)}\right)^k
\end{equation}
and by applying Taylor series expansion about $\zco r_i=\deg(v_i)$, we obtain 
\begin{equation}\label{row sum exponential}
\exp(r_i)=\sum_{k=0}^{\deg(v_i)}\frac{\exp(\deg v_i)}{k!}\zdu(r_i)^k.
\end{equation}

\subsection{Laplacian Eigenvalues}

\begin{lemma}
Let $\Lambda$ be the zeon combinatorial Laplacian of a simple connected finite graph $G=(V,E)$ on $m$ vertices, and suppose that for some positive integer $k$, $G$ has a unique vertex $v_0$ of degree $k$.  Then $\Lambda$ has a  spectrally simple eigenvalue $\lambda$  whose scalar part is $k$; i.e., $\zco \lambda=\deg(v_0)$.
\end{lemma}

\begin{proof}
First, we note that
\begin{align*}
\zco\chi_{\Lambda}(t)=\zco \vert t\mI-\Lambda\vert&=\vert t\mI-\zco\Lambda\vert\\
&=\chi_{\zco \Lambda}(t)\\
&=\prod_{v\in V}^m (t-\deg(v))\\
&=(t-\deg(v_0))\prod_{v\ne v_0}(t-\deg(v)).
\end{align*}
Hence, $\lambda_0=\deg(v_0)$ is a simple zero of $\zco\chi_\Lambda(u)$, implying the existence of spectrally simple zero $\lambda=\lambda_0+\zdu \lambda$ of the characteristic polynomial $\chi_{\Lambda}(u)$.  
\end{proof}

\begin{theorem}\label{laplacian eigenvalue}
Let $G=(V,E)$ be a simple finite graph on $m$ vertices $\{v_1, \ldots, v_m\}$ and suppose that for some positive integer $d$, the graph has a unique vertex $v_k$ of degree $d$.  It follows that the zeon combinatorial Laplacian $\Lambda$ of $G$ has a spectrally simple eigenvalue $\lambda$ satisfying \begin{equation*}
\lambda = d + \sum_{I\subseteq [m]}\omega_I\zeta_{I}\prod_{j\in I\setminus\{k\}}(d-\deg(v_j))^{-1},
\end{equation*}
where $(-1)^{|I|}\omega_I$ is the number of cycles based at $v_k$ on vertices indexed by $I$.
\end{theorem}

\begin{proof}
Begin by relabeling the vertices $\{v_1, \ldots, v_m\}$ if necessary so that $v_m$ is the unique vertex of degree $d$, and let $\Lambda$ be the zeon combinatorial Laplacian of $G$ after relabeling.  Since $v_m$ is the unique vertex of degree $d$, the first $m-1$ main diagonal elements of $\lambda\mI-\Lambda$ are invertible, while the last main diagonal element is nilpotent.   In particular, letting $d_j$ denote the degree of vertex $v_j$ and writing $\ell_j=\lambda-d_j$ for $j=1, \ldots, m-1$  , we have
\begin{equation}\label{form to triangularize}
\lambda\mI-\Lambda=\left(\begin{matrix}
\ell_1&\varepsilon_{12}\zeta_{\{2\}}&\cdots&\cdots&\varepsilon_{1m}\zeta_{\{m\}}\\
\varepsilon_{12}\zeta_{\{1\}}&\ell_2&\varepsilon_{13}\zeta_{\{3\}}&\cdots&\varepsilon_{2m}\zeta_{\{m\}}\\
\varepsilon_{13}\zeta_{\{1\}}&\varepsilon_{23}\zeta_{\{2\}}&\ell_3&\cdots&\varepsilon_{3m}\zeta_{\{m\}}\\
\vdots&\vdots&\vdots&\ddots&\vdots\\
\varepsilon_{1m}\zeta_{\{1\}}&\cdots&\cdots&\cdots&\zdu \lambda
\end{matrix}\right),  
 \end{equation}
where $\varepsilon_{ij}$ denotes the function $V\times V\to \{0,1\}$ defined by 
\begin{equation*}
\varepsilon_{ij}=\begin{cases}
1&\text{\rm if }\{v_i,v_j\}\in E,\\
0&\text{\rm otherwise.}
\end{cases}
\end{equation*}

It follows that by simply adding zeon multiples of rows to each other, $\lambda\mI-\Lambda$ can be placed into upper triangular form  \begin{equation}\label{triangular form}
\tau(\lambda\mI-\Lambda)=\left(\begin{matrix}
\alpha_1&*&\cdots&*&\eta_1\\
0&\alpha_2&\cdots&*&\eta_2\\
0&0&\ddots&\vdots&\vdots\\
\vdots&\vdots&0&\alpha_{m-1}&\eta_{m-1}\\
0&0&\cdots&0&\eta_m
\end{matrix}\right),  
 \end{equation}
where $\alpha_\ell\in\CZ^\times$ and $\eta_\ell\in\CZ^\circ$ for $\ell=1, \ldots, m$.  Further, since the only invertible elements of $\lambda\mI-\Lambda$ lie along the main diagonal, the only nilpotent column of $\lambda\mI-\Lambda$ is the $m$th column.   Since zeros are obtained below each element along the main diagonal by adding a zeon multiple of one row to another, the determinant of the resulting matrix is unchanged.
Hence, 
\begin{equation*}
0=\vert \lambda\mI-\Lambda\vert=\vert \tau(\lambda\mI-\Lambda)\vert=\eta_m\prod_{\ell=1}^{m-1}\alpha_\ell,
\end{equation*}
implying $\eta_m=0$.

Now set $\tau_0=\lambda\mI-\Lambda$ and let $\tau_k$ be the matrix obtained after $k$ steps of the triangularization process for $k=1, \ldots, m-1$.  

\begin{description}
\item [{\bf Claim 1.}] The off-diagonal elements of $\tau_k$ are given by the following:
\begin{equation}\label{off-diag}
(\tau_k)_{ij}=\begin{cases}
0&  i>j, j\le k;\\
\varepsilon_{ij}\zeta_{\{j\}}-\sum_{I\subseteq [k]}\gamma_{I\cup\{j\}}\zeta_{I\cup\{j\}}\prod_{t\in I\cap[k]}{\ell_t}^{-1}& i > j > k;\\
\varepsilon_{ij}\zeta_{\{j\}}-\sum_{I\subseteq [k-1]}\gamma_{I\cup\{j\}}\zeta_{I\cup\{j\}}\prod_{t\in I\cap[i-1]}{\ell_t}^{-1}&  i< j, j>k;\\
\varepsilon_{ij}\zeta_{\{j\}}-\sum_{I\subseteq [i-1]}\gamma_{I\cup\{j\}}\zeta_{I\cup\{j\}}\prod_{t\in I\cap[i-1]}{\ell_t}^{-1} &i<j\le k;
\end{cases}
\end{equation}where $(-1)^{|I\cup\{j\}|}\gamma_{I\cup\{j\}}$ is the number of paths from $v_i$ to $v_j$ on vertices indexed by $I\cup\{j\}$ (excluding the initial vertex $v_i$; i.e.,  $i\notin I\cup\{j\}$).  
\end{description}

Proof of the claim is by induction on $k$.  When $k=0$, the result holds by definition of $\tau_0$.  Now suppose \eqref{off-diag} is valid for all off-diagonal entries of $(\tau_{k-1})$.   
We now proceed by cases.

\begin{description}
\item [{\bf Case 1. ($i>j$ and $j\le k$)}]
Zeros are obtained below $(\tau_{k-1})_{kk}$ via Gaussian elimination using the inverse $((\tau_{k-1})_{kk})^{-1}$ in the usual way since we assume all diagonal entries (except possibly the last one) are invertible.  This gives $(\tau_k)_{ij}=0$ whenever $j<i$ and $j\le k$.
\end{description}

\begin{description}
\item [{\bf Case 2. ($i>j> k$)}]

We consider now those entries below the main diagonal that have not yet been zeroed out.  If $v_i$ is not adjacent to $v_k$, then $(\tau_{k-1})_{ik}=0$, so that $(\tau_k)_{ij}=(\tau_{k-1})_{ij}$; i.e., no change is made to row $i$ in the triangularization $\tau_{k-1}\to\tau_k$.  On the other hand, if $v_i$ is adjacent to $v_k$, then $\varepsilon_{ik}\ne 0$ and \begin{equation}
(\tau_k)_{ij}=(\tau_{k-1})_{ij}-\frac{(\tau_{k-1})_{ik}}{(\tau_{k-1})_{kk}}(\tau_{k-1})_{kj}.
\end{equation}  
Writing $(\tau_{k-1})_{kk}=\ell_k-\xi$, where $\xi\in\CZ^\circ$ is a $\CZ$-linear combination of elements annihilated by $\zeta_{\{k\}}$, it follows that $((\tau_{k-1})_{kk})^{-1}=\frac{1}{\ell_k}+\frac{1}{{\ell_k}^2}\xi$.  
Further, by the inductive hypothesis, $(\tau_{k-1})_{ik}$ is a $\CZ$-linear combination of elements annihilated by $\zeta_{\{k\}}$. Consequently, \begin{equation*}
\frac{(\tau_{k-1})_{ik}}{(\tau_{k-1})_{kk}}=(\tau_{k-1})_{ik}\left(\frac{1}{\ell_k}+\frac{1}{{\ell_k}^2}\xi\right)=\frac{(\tau_{k-1})_{ik}}{\ell_k}.
\end{equation*}
Hence, \begin{align*}
(\tau_k)_{ij}&=(\tau_{k-1})_{ij}-\frac{(\tau_{k-1})_{ik}}{\ell_k}(\tau_{k-1})_{kj}\\
&=\varepsilon_{ij}\zeta_{\{j\}}-\sum_{I\subseteq [k-1]}\gamma_{I\cup\{j\}}\zeta_{I\cup\{j\}}\prod_{t\in I\cap[k-1]}{\ell_t}^{-1}\\
&\hskip10pt-{\ell_k}^{-1}(\tau_{k-1})_{ik}(\tau_{k-1})_{kj}\\
&=\varepsilon_{ij}\zeta_{\{j\}}-\sum_{I\subseteq [k-1]}\gamma_{I\cup\{j\}}\zeta_{I\cup\{j\}}\prod_{t\in I\cap[k-1]}{\ell_t}^{-1}\\
&\hskip10pt-{\ell_k}^{-1}\left(\varepsilon_{ik}\zeta_{\{k\}}-\sum_{I\subseteq [k-1]}{\gamma_{I\cup\{k\}}}'\zeta_{I\cup\{k\}}\prod_{t\in I\cap[k-1]}{\ell_t}^{-1}\right)\\
&\hskip10pt\times\left(\varepsilon_{kj}\zeta_{\{j\}}-\sum_{I\subseteq [k-1]}{\gamma_{I\cup\{j\}}}''\zeta_{I\cup\{j\}}\prod_{t\in I\cap[k-1]}{\ell_t}^{-1}\right),
\end{align*}
where $(-1)^{|I\cup\{j\}|}\gamma_{I\cup\{j\}}$ is the number of paths from $v_i$ to $v_j$ on vertices indexed by $I\cup\{j\}$, $(-1)^{|I|\cup\{k\}}{\gamma_{I\cup\{k\}}}'$ is the number of paths from $v_i$ to $v_k$ on vertices indexed by $I\cup\{k\}$, and $(-1)^{|I|\cup\{j\}}{\gamma_{I\cup\{j\}}}''$ is the number of paths from $v_k$ to $v_j$ on vertices indexed by $I\cup\{j\}$, respectively.

Expanding the right-hand side of the final equality, we obtain 
\begin{align*}
(\tau_k)_{ij}&=\varepsilon_{ij}\zeta_{\{j\}}-\sum_{I\subseteq [k-1]}\gamma_{I\cup\{j\}}\zeta_{I\cup\{j\}}\prod_{t\in I\cap[k-1]}{\ell_t}^{-1}\\
&\hskip10pt-{\ell_k}^{-1}\varepsilon_{ik}\varepsilon_{kj}\zeta_{\{j,k\}}\\
&\hskip10pt+{\ell_k}^{-1}\varepsilon_{ik}\zeta_{\{k\}}\sum_{I\subseteq [k-1]}{\gamma_{I\cup\{j\}}}''\zeta_{I\cup\{j\}}\prod_{t\in I\cap[k-1]}{\ell_t}^{-1}\\
&\hskip10pt+{\ell_k}^{-1}\varepsilon_{kj}\zeta_{\{j\}}\sum_{I\subseteq [k-1]}{\gamma_{I\cup\{k\}}}'\zeta_{I\cup\{k\}}\prod_{t\in I\cap[k-1]}{\ell_t}^{-1}\\
&\hskip10pt-{\ell_k}^{-1}\left[\sum_{I\subseteq [k-1]}{\gamma_{I\cup\{k\}}}'\zeta_{I\cup\{k\}}\prod_{t\in I\cap[k-1]}{\ell_t}^{-1}\right]\\
&\hskip10pt\times\left[\sum_{I\subseteq [k-1]}{\gamma_{I\cup\{j\}}}''\zeta_{I\cup\{j\}}\prod_{t\in I\cap[k-1]}{\ell_t}^{-1}\right].
\end{align*}

Note that $\varepsilon_{ij}\zeta_{\{j\}}$ represents the unique one-step path from $v_i$ to $v_j$ when $v_i$ and $v_j$ are adjacent.  Further, $-{\ell_k}^{-1}\varepsilon_{ik}\varepsilon_{kj}\zeta_{\{j,k\}}$ represents the unique two-step path $v_i\to v_k\to v_j$ when $\{v_i,v_k\}, \{v_k,v_j\}\in E$.  

The quantity \begin{equation*}
+{\ell_k}^{-1}\varepsilon_{ik}\zeta_{\{k\}}\sum_{I\subseteq [k-1]}{\gamma_{I\cup\{j\}}}''\zeta_{I\cup\{j\}}\prod_{t\in I\cap[k-1]}{\ell_t}^{-1}
\end{equation*}
represents paths $v_i\to v_j$ on vertices indexed by subsets of $[k]$ passing through vertex $v_k$ in the second step.

The quantity \begin{equation*}
{\ell_k}^{-1}\varepsilon_{kj}\zeta_{\{j\}}\sum_{I\subseteq [k-1]}{\gamma_{I\cup\{k\}}}'\zeta_{I\cup\{k\}}\prod_{t\in I\cap[k-1]}{\ell_t}^{-1}
\end{equation*}
represents paths $v_i\to v_j$ on vertices indexed by subsets of $[k]$ passing through vertex $v_k$ in the penultimate step.

Finally, the quantity 
\begin{align*}
-{\ell_k}^{-1}\left[\sum_{I\subseteq [k-1]}{\gamma_{I\cup\{k\}}}'\zeta_{I\cup\{k\}}\prod_{t\in I\cap[k-1]}{\ell_t}^{-1}\right]\\
 \times
\left[\sum_{I\subseteq [k-1]}{\gamma_{I\cup\{j\}}}''\zeta_{I\cup\{j\}}\prod_{t\in I\cap[k-1]}{\ell_t}^{-1}\right]
\end{align*}
represents the product of walks $v_i\to v_k$ with walks from $v_k\to v_j$ on vertices indexed by a subset of $[k]$ which pass though $v_k$ in some intermediate step.

\end{description}

\begin{description}
\item [{\bf Case 3. ($i<j$ and $j>k$)}]

Considering entries $(\tau_k)_{ij}$ above the main diagonal with $j>k$,  the proof proceeds in similar fashion to that of Case 2.   In particular, 
 \begin{align*}
(\tau_k)_{ij}&=(\tau_{k-1})_{ij}-\frac{(\tau_{k-1})_{ik}}{\ell_k}(\tau_{k-1})_{kj}\\
&=\varepsilon_{ij}\zeta_{\{j\}}-\sum_{I\subseteq [k-2]}\gamma_{I\cup\{j\}}\zeta_{I\cup\{j\}}\prod_{t\in I\cap[i-1]}{\ell_t}^{-1}\\
&\hskip10pt-{\ell_k}^{-1}(\tau_{k-1})_{ik}(\tau_{k-1})_{kj}\\
&=\varepsilon_{ij}\zeta_{\{j\}}-\sum_{I\subseteq [k-2]}\gamma_{I\cup\{j\}}\zeta_{I\cup\{j\}}\prod_{t\in I\cap[i-1]}{\ell_t}^{-1}\\
&\hskip10pt-{\ell_k}^{-1}\left(\varepsilon_{ik}\zeta_{\{k\}}-\sum_{I\subseteq [i-1]}{\gamma_{I\cup\{k\}}}'\zeta_{I\cup\{k\}}\prod_{t\in I\cap[i-1]}{\ell_t}^{-1}\right)\\
&\hskip10pt\times\left(\varepsilon_{kj}\zeta_{\{j\}}-\sum_{I\subseteq [k-2]}{\gamma_{I\cup\{j\}}}''\zeta_{I\cup\{j\}}\prod_{t\in I\cap[k-1]}{\ell_t}^{-1}\right),
\end{align*}
where $(-1)^{|I\cup\{j\}|}\gamma_{I\cup\{j\}}$ is the number of paths and PWICs from $v_i$ to $v_j$ on vertices indexed by $I\cup\{j\}$, $(-1)^{|I|\cup\{k\}}{\gamma_{I\cup\{k\}}}'$ is the number of paths and PWICs from $v_i$ to $v_k$ on vertices indexed by $I\cup\{k\}$, and $(-1)^{|I|\cup\{j\}}{\gamma_{I\cup\{j\}}}''$ is the number of paths and PWICs from $v_k$ to $v_j$ on vertices indexed by $I\cup\{j\}$, respectively.
Expanding the right-hand side of the final equality, we obtain 
\begin{align*}
(\tau_k)_{ij}&=\varepsilon_{ij}\zeta_{\{j\}}-\sum_{I\subseteq [k-2]}\gamma_{I\cup\{j\}}\zeta_{I\cup\{j\}}\prod_{t\in I\cap[i-1]}{\ell_t}^{-1}\\
&\hskip10pt-{\ell_k}^{-1}\varepsilon_{ik}\varepsilon_{kj}\zeta_{\{j,k\}}\\
&\hskip10pt+{\ell_k}^{-1}\varepsilon_{ik}\zeta_{\{k\}}\sum_{I\subseteq [k-2]}{\gamma_{I\cup\{j\}}}''\zeta_{I\cup\{j\}}\prod_{t\in I\cap[k-1]}{\ell_t}^{-1}\\
&\hskip10pt+{\ell_k}^{-1}\varepsilon_{kj}\zeta_{\{j\}}\sum_{I\subseteq [i-1]}{\gamma_{I\cup\{k\}}}'\zeta_{I\cup\{k\}}\prod_{t\in I\cap[i-1]}{\ell_t}^{-1}\\
&\hskip10pt-{\ell_k}^{-1}\left[\sum_{I\subseteq [k-2]}{\gamma_{I\cup\{k\}}}'\zeta_{I\cup\{k\}}\prod_{t\in I\cap[k-1]}{\ell_t}^{-1}\right]\\
&\hskip10pt\times\left[\sum_{I\subseteq [i-1]}{\gamma_{I\cup\{j\}}}''\zeta_{I\cup\{j\}}\prod_{t\in I\cap[i-1]}{\ell_t}^{-1}\right].
\end{align*}

Recall that $\varepsilon_{ij}\zeta_{\{j\}}$ represents the unique one-step path from $v_i$ to $v_j$ when $v_i$ and $v_j$ are adjacent.  Further, $-{\ell_k}^{-1}\varepsilon_{ik}\varepsilon_{kj}\zeta_{\{j,k\}}$ represents the unique two-step path $v_i\to v_k\to v_j$ when $\{v_i,v_k\}, \{v_k,v_j\}\in E$.  

Since $(-1)^{|I|\cup\{j\}}{\gamma_{I\cup\{j\}}}''$ is the number of paths and PWICs $v_k\to v_j$, and $\varepsilon_{ik}=1$ if and only if $v_i$ is adjacent to $v_k$, the quantity \begin{equation*}
{\ell_k}^{-1}\varepsilon_{ik}\zeta_{\{k\}}\sum_{I\subseteq [k-2]}{\gamma_{I\cup\{j\}}}''\zeta_{I\cup\{j\}}\prod_{t\in I\cap[k-1]}{\ell_t}^{-1}
\end{equation*}
represents paths and PWICs $v_i\to v_j$ on vertices indexed by subsets of $[k]$ passing through vertex $v_k$ in the second step.

By similar reasoning, the quantity \begin{equation*}
{\ell_k}^{-1}\varepsilon_{kj}\zeta_{\{j\}}\sum_{I\subseteq [i-1]}{\gamma_{I\cup\{k\}}}'\zeta_{I\cup\{k\}}\prod_{t\in I\cap[i-1]}{\ell_t}^{-1}\end{equation*}
represents paths $v_i\to v_j$ on vertices indexed by subsets of $[k]$ passing through vertex $v_k$ in the penultimate step.

Finally, the quantity 
\begin{align*}
-{\ell_k}^{-1}\left[\sum_{I\subseteq [k-2]}{\gamma_{I\cup\{k\}}}'\zeta_{I\cup\{k\}}\prod_{t\in I\cap[k-1]}{\ell_t}^{-1}\right]\\
\times\left[\sum_{I\subseteq [i-1]}{\gamma_{I\cup\{j\}}}''\zeta_{I\cup\{j\}}\prod_{t\in I\cap[i-1]}{\ell_t}^{-1}\right]
\end{align*}
represents the product of walks $v_i\to v_k$ with walks from $v_k\to v_j$ on vertices indexed by a subset of $[k]$ which pass though $v_k$ in some intermediate step.  Summing the quantities establishes the stated result for Case 3.
\end{description}

\begin{description}
\item [{\bf Case 4. ($i<j\le k$)}]
Case 4 follows from Case 3 by observing that once zeros are obtained below diagonal $(\tau_k)_{kk}$, rows 1 through $k$ remain fixed; i.e., $(\tau_n)_{ij}=(\tau_k)_{ij}$ for $1\le i\le k$ and $k\le n\le m$.   Hence, the only values ${\ell_t}^{-1}$ appearing in terms of $(\tau_k)_{ij}$ are indexed by $t\in [i-1]$. 
\end{description}

This concludes the proof of Claim 1.  We now turn to elements along the main diagonal of $\tau_k$.

\begin{description}
\item [{\bf Claim 2.}] The diagonal elements of $\tau_k$ are given by the following:
\begin{equation}\label{diag elts}
(\tau_k)_{jj}=\begin{cases}\ell_j-\sum_{I\subseteq [k]}\omega_{I\cup\{j\}}\zeta_{I\cup\{j\}}\prod_{t\in I\cap[k]}{\ell_t}^{-1}&j>k,\\
\ell_j-\sum_{I\subseteq [j]}{\omega_I}'\zeta_I\prod_{t\in I\cap[j-1]}{\ell_t}^{-1}&j\le k,
\end{cases}
\end{equation} where $(-1)^{|I\cup\{j\}|}\omega_{I\cup\{j\}}$ is the number of cycles based at $v_j$ on vertices indexed by $I\cup\{j\}$ and $(-1)^{|I|}{\omega_I}'$ is the number of cycles based at $v_j$ on vertices indexed by $I$.  
\end{description}

Proof of this claim is again by induction on $k$.  When $k=0$, the result holds by definition of $\tau_0$.  Now suppose \eqref{off-diag} is valid for all off-diagonal entries of $(t_{k-1})$ and that \eqref{diag elts} is valid for diagonal elements of $(\tau_{k-1})$.   If $v_k$ is not adjacent to $v_j$, we see that $(\tau_k)_{jj}=(\tau_{k-1})_{jj}$, and there is nothing to prove.  On the other hand, if $v_k$ is adjacent to $v_j$, we have 
\begin{equation}\label{diagonal recurrence}
(\tau_k)_{jj}=(\tau_{k-1})_{jj}-\frac{(\tau_{k-1})_{jk}}{(\tau_{k-1})_{kk}}(\tau_{k-1})_{kj}.
\end{equation}  

We now continue by cases.

\begin{description}
\item [{\bf Case 1. ($j>k$)}]
When $j>k$, \eqref{diagonal recurrence} becomes
\begin{align*}
(\tau_k)_{jj}&=\ell_j-\sum_{I\subseteq [k-1]\cup\{j\}}\omega_I\zeta_I\prod_{t\in I}{\ell_t}^{-1}-\frac{1}{\ell_k}(\tau_{k-1})_{jk}(\tau_{k-1})_{kj}
\end{align*}
where $(-1)^{|I|}\omega_I$ is the number of cycles based at $v_j$ on vertices indexed by $I\subseteq[k-1]\cup\{j\}$ and \eqref{off-diag} gives
\begin{align*}
(\tau_{k-1})_{jk}(\tau_{k-1})_{kj}&=\left(\zeta_{\{k\}}-\sum_{I\subseteq [k-1]\cup\{k\}}\gamma_I\zeta_I\prod_{t\in I\setminus\{k-1\}}{\ell_t}^{-1}\right)\\
&\hskip10pt\times\left(\zeta_{\{j\}}-\sum_{I\subseteq [k-1]\cup\{j\}}{\gamma_I}'\zeta_I\prod_{t\in I\setminus\{k-1\}}{\ell_t}^{-1}\right)\\
&=\zeta_{\{j,k\}}-\zeta_{\{k\}}\sum_{I\subseteq [k-1]\cup\{j\}}{\gamma_I}'\zeta_I\prod_{t\in I\setminus\{k-1\}}{\ell_t}^{-1}\\
&\hskip10pt -\zeta_{\{j\}}\sum_{I\subseteq [k-1]\cup\{k\}}\gamma_I\zeta_I\prod_{t\in I\setminus\{k-1\}}{\ell_t}^{-1}\\
&\hskip10pt+\left(\sum_{I\subseteq [k-1]\cup\{k\}}\gamma_I\zeta_I\prod_{t\in I\setminus\{k-1\}}{\ell_t}^{-1}\right)\\
&\hskip10pt\times\left(\sum_{I\subseteq [k-1]\cup\{j\}}{\gamma_I}'\zeta_I\prod_{t\in I\setminus\{k-1\}}{\ell_t}^{-1}\right),
\end{align*}
where $(-1)^{|I|}\gamma_I$ denotes the number of paths from $v_j$ to $v_k$ on vertices indexed by $I\subseteq [k]$ and $(-1)^{|I|}{\gamma_I}'$ denotes the number of paths from $v_k$ to $v_j$ on vertices indexed by $I\subseteq [k-1]\cup \{j\}$.  Now $\zeta_{\{j,k\}}$ represents a 2-cycle on vertices $\{v_j, v_k\}$.  The quantity \begin{equation*}
\zeta_{\{k\}}\sum_{I\subseteq [k-1]\cup\{j\}}{\gamma_I}'\zeta_I\prod_{t\in I\setminus\{k-1\}}{\ell_t}^{-1}=\sum_{I\subseteq [k]\cup\{j\}}{\gamma_I}'\zeta_I\prod_{t\in I\setminus\{k-1\}}{\ell_t}^{-1}
\end{equation*}
represents cycles based at $v_j$ whose first step is to $v_j\to v_k$.  Similarly, 
\begin{equation*}
\zeta_{\{j\}}\sum_{I\subseteq [k-1]\cup\{k\}}\gamma_I\zeta_I\prod_{t\in I\setminus\{k-1\}}{\ell_t}^{-1}
\end{equation*}
represents cycles based at $v_j$ whose last step is $v_k\to v_j$.
Finally, fixing $I\subseteq[k]$ and $J\subseteq[k-1]\cup\{j\}$,  the nonzero terms of the product of last terms are of the form
\begin{align}
\left[\gamma_I\zeta_I\prod_{t\in I\setminus\{k-1\}}{\ell_t}^{-1}\right]
\left[{\gamma_J}'\zeta_J\prod_{s\in J\setminus\{k-1\}}{\ell_s}^{-1}\right]\nonumber\\
=\gamma_I{\gamma_J}'\zeta_{I\cup J}\prod_{r\in (I\cup J)\setminus\{k-1\}}{\ell_r}^{-1},\end{align}
provided $I\cap J=\varnothing$.  This represents all remaining cycles $v_j\to v_k\to v_j$ on vertices indexed by $I\cup J$.  Now $(-1)^{|I\cup J|}\gamma_I{\gamma_J}'$ denotes the number of cycles based at $v_j$ on vertices indexed by $I\cup J$ including $v_k$ at an intermediate step.  
\end{description}

\begin{description}
\item [{\bf Case 2. ($j\le k$)}]
Recalling that $(\tau_\ell)_{ij}=(\tau_k)_{ij}$ for $1\le \ell\le k$, the only values ${\ell_t}^{-1}$ appearing in terms of $(\tau_k)_{jj}$ are indexed by $t\in [j-1]$.  Otherwise, the proof proceeds as the proof of Case 1.
\end{description}
This concludes the proof of Claim 2.  Letting $k=m-1$ and $j=m$ in \eqref{diag elts}, we now see that  \begin{equation}\label{last diagonal}
(\tau_{m-1})_{mm}=\ell_m-\sum_{I\subseteq [m]}\omega_I\zeta_I\prod_{j\in I\setminus\{m\}}{\ell_j}^{-1},
\end{equation}
where $(-1)^{|I|}\omega_I$ is the number of cycles based at $v_m$ on vertices indexed by $I\subseteq [m]$.  In particular, $m\in I$ for each nonzero term of \eqref{last diagonal}.

Referring to \eqref{triangular form}, each $\eta_j$ is a $\CZ$-linear combination of elements annihilated by $\zeta_{\{m\}}$.  Further, $\eta_m=0$ implies $\zdu\lambda$ is a $\CZ$-linear combination of $\eta_j$s, so $\zeta_{\{m\}}$ annihilates $\zdu\lambda$.  Consequently,  $(\zdu \lambda)^2=0$, and $\eta_j\zdu\lambda=0$ for $j=1, \ldots, m-1$.  Moreover,  it follows that \begin{equation*}
{\ell_j}^{-1}=\frac{1}{d-d_j}-\frac{1}{(d-d_j)^2}\zdu\lambda,
\end{equation*}
and that $\zeta_I\zdu\lambda =0$ when $m\in I$. 

The proof of the theorem is now completed by recalling that $\ell_m=\zdu \lambda$ and $(\tau_{m-1})_{mm}=0$.   In particular, 
\begin{align*}
\zdu\lambda&=\sum_{I\subseteq [m]}\omega_I\zeta_I\prod_{j\in I\setminus\{m\}}{\ell_j}^{-1}\\
&=\sum_{I\subseteq [m]}\omega_I\zeta_{I}\prod_{j\in I\setminus\{m\}}\left(\frac{1}{d-d_j}-\frac{1}{(d-d_j)^2}\zdu\lambda\right)\\
&=\sum_{I\subseteq [m]}\omega_I\zeta_{I}\prod_{j\in I\setminus\{m\}}\frac{1}{d-d_j}\\
&=\lambda-d,
\end{align*}
where $(-1)^{|I|}\omega_I$ is the number of cycles based at $v_m$ on vertices indexed by $I\subseteq [m]$.  

\end{proof}

\begin{example}
\begin{figure}
\fbox{\includegraphics[width=120pt]{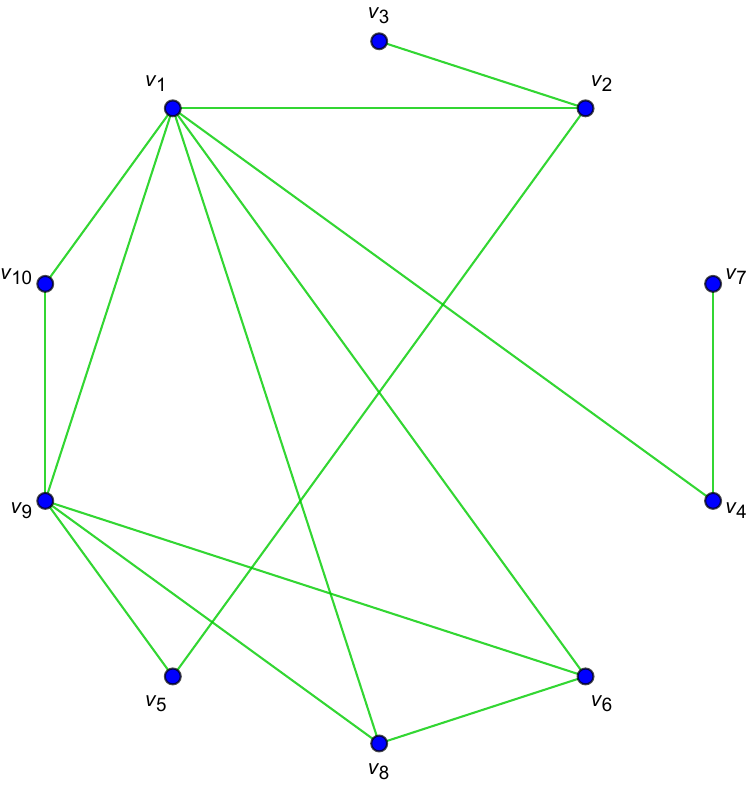}\hskip5pt\includegraphics[width=200pt]{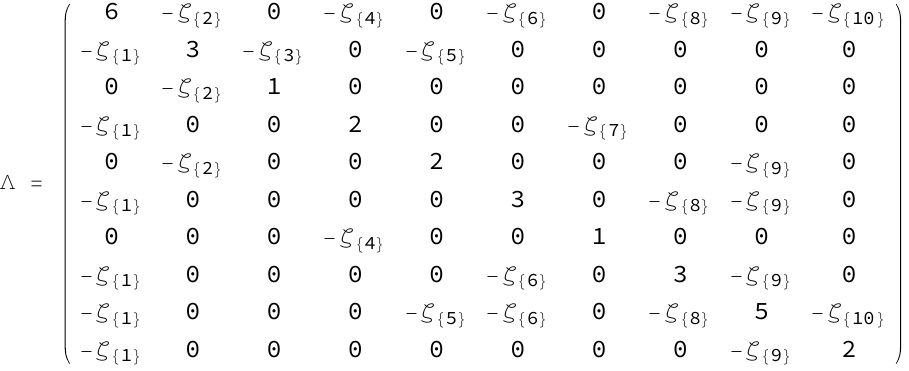}}
\caption{A 10-vertex graph and its zeon Laplacian.\label{10vgraph}}
\end{figure}

The graph of Figure \ref{10vgraph} has a unique vertex, $v_1$, of degree 6 and a unique vertex, $v_9$, of degree 5.   The eigenvalues of these vertices, denoted by $\lambda_6$ and $\lambda_5$, respectively, are as follows:

\begin{align*}
\lambda_6&=6+\frac{1}{3} \zeta _{\{1,2\}}+\frac{1}{4} \zeta _{\{1,4\}}+\frac{1}{3} \zeta _{\{1,6\}}+\frac{1}{3} \zeta _{\{1,8\}}+\zeta _{\{1,9\}}+\frac{1}{4} \zeta _{\{1,10\}}
-\frac{2}{9} \zeta _{\{1,6,8\}}\\
&-\frac{2}{3} \zeta _{\{1,6,9\}}-\frac{2}{3} \zeta _{\{1,8,9\}}-\frac{1}{2} \zeta _{\{1,9,10\}}+\frac{1}{6} \zeta _{\{1,2,5,9\}}+\frac{2}{3} \zeta _{\{1,6,8,9\}}+\frac{1}{6} \zeta _{\{1,6,9,10\}}\\
&+\frac{1}{6} \zeta _{\{1,8,9,10\}}-\frac{1}{18} \zeta _{\{1,2,5,6,9\}}-\frac{1}{18} \zeta _{\{1,2,5,8,9\}}-\frac{1}{24} \zeta _{\{1,2,5,9,10\}}\\
&-\frac{1}{9} \zeta _{\{1,6,8,9,10\}}+\frac{1}{27} \zeta _{\{1,2,5,6,8,9\}},
\end{align*}
and 
\begin{align*}
\lambda_5&=5-\zeta _{\{1,9\}}+\frac{1}{3} \zeta _{\{5,9\}}+\frac{1}{2} \zeta _{\{6,9\}}+\frac{1}{2} \zeta _{\{8,9\}}+\frac{1}{3} \zeta _{\{9,10\}}+\zeta _{\{1,6,9\}}+\zeta _{\{1,8,9\}}\\
&+\frac{2}{3} \zeta _{\{1,9,10\}}-\frac{1}{2} \zeta _{\{6,8,9\}}-\frac{1}{3} \zeta _{\{1,2,5,9\}}-\frac{3}{2} \zeta _{\{1,6,8,9\}}-\frac{1}{3} \zeta _{\{1,6,9,10\}}\\
&-\frac{1}{3} \zeta _{\{1,8,9,10\}}+\frac{1}{6} \zeta _{\{1,2,5,6,9\}}+\frac{1}{6} \zeta _{\{1,2,5,8,9\}}+\frac{1}{9} \zeta _{\{1,2,5,9,10\}}+\frac{1}{3} \zeta _{\{1,6,8,9,10\}}\\
&-\frac{1}{6} \zeta _{\{1,2,5,6,8,9\}}.
\end{align*}

\end{example}

\subsubsection{Eigenvalues of $\Lambda=\Delta-\Psi$ via Exponentials of $\Psi$}

Having characterized eigenvalues of the zeon combinatorial Laplacian in terms of the associated graph's cycles, an alternative formulation is available.  

\begin{corollary}
Let $G=(V,E)$ be a simple finite graph and suppose that for some positive integer $k$, the graph has a unique vertex $v$ of degree $k$.  Let $\Lambda=\Delta-\Psi$ be the zeon combinatorial Laplacian of $G$, where $\Psi$ is the nilpotent adjacency matrix of $G$,  and let $\lambda$ be the eigenvalue of $\Lambda$ satisfying $\zco \lambda =k$.  Then,
\begin{equation*}
\lambda = k + \sum_{I\subseteq V}\omega_I\zeta_{I}\prod_{v_j\in I\setminus\{v\}}(k-\deg(v_j))^{-1},
\end{equation*}
where \begin{equation*}
(-1)^{|I|}\omega_I=|I|!\left< \langle v\vert \exp(\Psi)\vert v\rangle,\zeta_I\right>.
\end{equation*}
\end{corollary}

\begin{proof}
Observing that $\Psi$ is nilpotent of index $\kappa(\Psi)$, we can write \hfill\break $
\exp(\Psi)=\displaystyle\sum_{n=0}^{\kappa(\Psi)}\frac{1}{n!}\Psi^n.$
By Theorem \ref{nil-structure theorem} (nil-structure), \begin{equation*}
\left< \langle v\vert \exp(\Psi)\vert v\rangle,\zeta_I\right>=\frac{1}{|I|!}\omega_I\zeta_I,
\end{equation*}
where $\omega_I$ is the number of cycles based at $v$ on vertices indexed by $I$.
\end{proof}

In fact, the expansion of the eigenvalue $\lambda$ associated with vertex $v$ can be expressed in terms of the associated main diagonal element of the matrix exponential $e^\Psi$ of the nilpotent adjacency matrix.

\begin{corollary}\label{lambda from Psi}
Let $\varkappa_v=\prod_{j\in I\setminus\{v\}}(\deg(v)-\deg(v_j))$.  Then,
\begin{equation}
\langle \lambda, \zeta_I\rangle=(-1)^{|I|}|I|!\left<\langle v\vert e^\Psi \vert v\rangle,\zeta_I\right>{\varkappa_v}^{-1}.
\end{equation}
Equivalently,
\begin{equation}
\left<\langle v\vert e^\Psi \vert v\rangle,\zeta_I\right>=\frac{(-1)^{|I|}\langle \lambda, \zeta_I\rangle}{|I|!}\varkappa_v.
\end{equation} 
\end{corollary}

\begin{figure}
$$\fbox{\includegraphics[width=160pt]{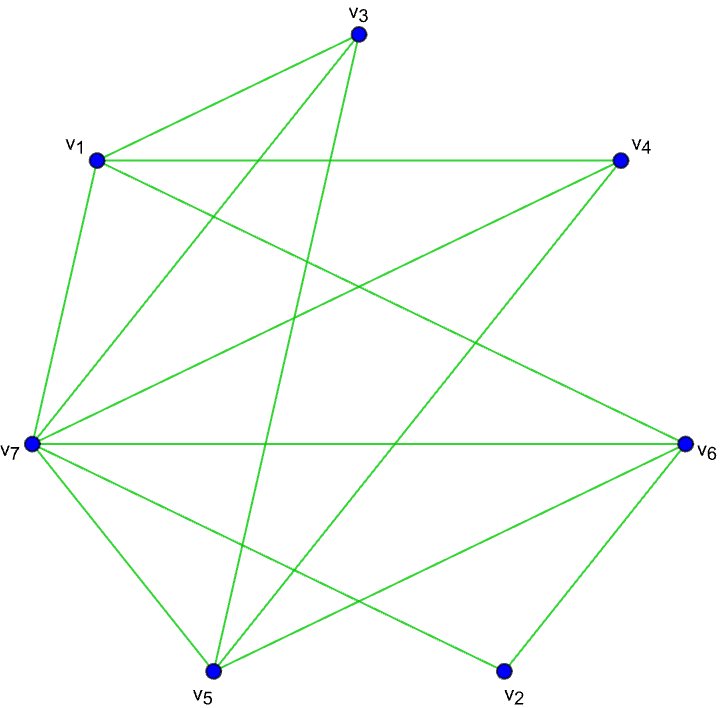}\,\,\,\includegraphics[width=150pt]{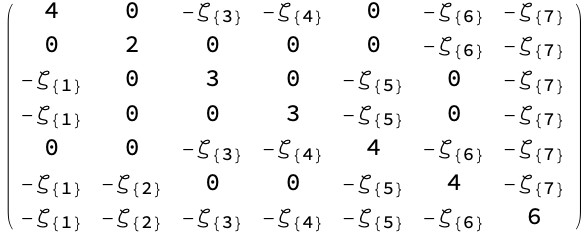}}$$
\caption{A seven vertex simple graph and its zeon Laplacian. \label{lambdaPsigraph}}
\end{figure}

Example \ref{lambda vs e^psi} illustrates the result of Corollary \ref{lambda from Psi}.
 
\begin{example}\label{lambda vs e^psi}
The zeon eigenvalue associated with vertex $v_7$ of the graph in Figure \ref{lambdaPsigraph} is 
\begin{align*}
\lambda&=6+\frac{1}{2} \zeta _{\{1,7\}}+\frac{1}{4} \zeta _{\{2,7\}}+\frac{1}{3} \zeta _{\{3,7\}}+\frac{1}{3} \zeta _{\{4,7\}}+\frac{1}{2} \zeta _{\{5,7\}}+\frac{1}{2} \zeta _{\{6,7\}}\\
&-\frac{1}{3} \zeta _{\{1,3,7\}}-\frac{1}{3} \zeta _{\{1,4,7\}}-\frac{1}{2} \zeta _{\{1,6,7\}}-\frac{1}{4} \zeta _{\{2,6,7\}}-\frac{1}{3} \zeta _{\{3,5,7\}}-\frac{1}{3} \zeta _{\{4,5,7\}}\\
&-\frac{1}{2} \zeta _{\{5,6,7\}}+\frac{1}{8} \zeta _{\{1,2,6,7\}}+\frac{1}{9} \zeta _{\{1,3,4,7\}}+\frac{1}{6} \zeta _{\{1,3,5,7\}}+\frac{1}{6} \zeta _{\{1,3,6,7\}}+\frac{1}{6} \zeta _{\{1,4,5,7\}}\\
&+\frac{1}{6} \zeta _{\{1,4,6,7\}}+\frac{1}{4} \zeta _{\{1,5,6,7\}}+\frac{1}{8} \zeta _{\{2,5,6,7\}}+\frac{1}{9} \zeta _{\{3,4,5,7\}}+\frac{1}{6} \zeta _{\{3,5,6,7\}}+\frac{1}{6} \zeta _{\{4,5,6,7\}}\\
&-\frac{1}{24} \zeta _{\{1,2,3,6,7\}}-\frac{1}{24} \zeta _{\{1,2,4,6,7\}}-\frac{2}{9} \zeta _{\{1,3,4,5,7\}}-\frac{1}{3} \zeta _{\{1,3,5,6,7\}}-\frac{1}{3} \zeta _{\{1,4,5,6,7\}}\\
&-\frac{1}{24} \zeta _{\{2,3,5,6,7\}}-\frac{1}{24} \zeta _{\{2,4,5,6,7\}}+\frac{1}{24} \zeta _{\{1,2,3,5,6,7\}}+\frac{1}{24} \zeta _{\{1,2,4,5,6,7\}}\\
&+\frac{1}{6} \zeta _{\{1,3,4,5,6,7\}}-\frac{1}{36} \zeta _{\{1,2,3,4,5,6,7\}}.
\end{align*}
The last main diagonal element of $\exp(\Psi)$ is 
\begin{align*}
\langle 7\vert \exp(\Psi)\vert 7\rangle&=1+\frac{1}{2} \zeta _{\{1,7\}}+\frac{1}{2} \zeta _{\{2,7\}}+\frac{1}{2} \zeta _{\{3,7\}}+\frac{1}{2} \zeta _{\{4,7\}}+\frac{1}{2} \zeta _{\{5,7\}}\\
&+\frac{1}{2} \zeta _{\{6,7\}}+\frac{1}{3} \zeta _{\{1,3,7\}}+\frac{1}{3} \zeta _{\{1,4,7\}}+\frac{1}{3} \zeta _{\{1,6,7\}}+\frac{1}{3} \zeta _{\{2,6,7\}}\\
&+\frac{1}{3} \zeta _{\{3,5,7\}}+\frac{1}{3} \zeta _{\{4,5,7\}}+\frac{1}{3} \zeta _{\{5,6,7\}}+\frac{1}{12} \zeta _{\{1,2,6,7\}}+\frac{1}{12} \zeta _{\{1,3,4,7\}}\\
&+\frac{1}{12} \zeta _{\{1,3,5,7\}}+\frac{1}{12} \zeta _{\{1,3,6,7\}}+\frac{1}{12} \zeta _{\{1,4,5,7\}}+\frac{1}{12} \zeta _{\{1,4,6,7\}}\\
&+\frac{1}{12} \zeta _{\{1,5,6,7\}}+\frac{1}{12} \zeta _{\{2,5,6,7\}}+\frac{1}{12} \zeta _{\{3,4,5,7\}}+\frac{1}{12} \zeta _{\{3,5,6,7\}}\\
&+\frac{1}{12} \zeta _{\{4,5,6,7\}}+\frac{1}{60} \zeta _{\{1,2,3,6,7\}}+\frac{1}{60} \zeta _{\{1,2,4,6,7\}}+\frac{1}{15} \zeta _{\{1,3,4,5,7\}}\\
&+\frac{1}{15} \zeta _{\{1,3,5,6,7\}}+\frac{1}{15} \zeta _{\{1,4,5,6,7\}}+\frac{1}{60} \zeta _{\{2,3,5,6,7\}}+\frac{1}{60} \zeta _{\{2,4,5,6,7\}}\\
&+\frac{1}{180} \zeta _{\{1,2,3,5,6,7\}}+\frac{1}{180} \zeta _{\{1,2,4,5,6,7\}}+\frac{1}{60} \zeta _{\{1,3,4,5,6,7\}}\\
&+\frac{1}{630} \zeta _{\{1,2,3,4,5,6,7\}}.
\end{align*}

Note that letting $I=\{1,2,3,4,5,6,7\}$, we have $\langle\lambda, \zeta_I\rangle=-\frac{1}{36}$ and $\langle\langle 7\vert \exp(\Psi)\vert 7\rangle,\zeta_{I}\rangle=\frac{1}{630}$.  Further, in accordance with Corollary \ref{lambda from Psi},  \begin{align*}
\langle \lambda, \zeta_{I}\rangle&=-\frac{1}{36}\\
&=(-1)^7 7!\frac{1}{630}\frac{1}{2^3\cdot3^2\cdot4}\\
&=(-1)^{|I|}|I|!\langle\langle 7\vert \exp(\Psi)\vert 7\rangle, \zeta_{I}\rangle\prod_{j\in I\setminus \{7\}}\frac{1}{(6-\deg(v_j))}.
\end{align*}

\end{example}

It is now apparent that any eigenvalue of $\Lambda$ associated with a vertex of unique degree in a simple graph can be computed directly from the nilpotent adjacency matrix with no need to compute the characteristic polynomial. 

\subsection{Laplacian Eigenvectors}

Having characterized the zeon eigenvalue of $\Lambda$ corresponding to a vertex of unique degree in a graph $G$, we now turn to the task of describing the associated zeon eigenvector. 

\begin{theorem}\label{Laplacian eigenvector}
Let $G=(V,E)$ be a simple finite graph on $m$ vertices $V=\{v_1, \ldots, v_m\}$ and suppose that for some integer $d$, the graph has a unique vertex $v_r$ of degree $d$.  Let $\Lambda$ be the zeon combinatorial Laplacian of $G$, and let $\lambda$ be the unique eigenvalue of $\Lambda$ satisfying $\zco \lambda=d$.  Then
$\Lambda$ has an eigenvector of the form $x_\lambda=(\mu_1, \ldots,  \mu_{m})^\intercal$ where $\mu_r=1$ and $\mu_i\in{\CZ_m}^\circ$ for $i\ne r$.  In particular, for each $\ell\ne r$, \begin{equation}\label{eigenvector characterization 1}
\langle\mu_\ell, \zeta_I\rangle=\gamma_I(d-d_\ell)^{-1}\prod_{j\in I}(d-d_j)^{-1},
\end{equation}
where $(-1)^{|I|}\gamma_I$ denotes the number of paths and PWICs $v_\ell\to v_r$ on vertices indexed by $I$.   Moreover, \begin{equation*}
\lambda=d+\sum_{\{j: \{v_j,v_r\}\in E\}}\zeta_{\{j\}}\mu_j.
\end{equation*}
\end{theorem}

\begin{proof}
For convenience, relabel the vertices if necessary so that vertex $v_m$ is the unique vertex of degree $d$, let $\Lambda$ be the zeon combinatorial Laplacian of $G$, and let $\lambda$ be the unique eigenvalue of $\Lambda$ satisfying $\zco \lambda=d$.
It follows that the reduced row echelon form of $\lambda\mI-\Lambda$ is of the form \begin{equation}\label{RRF}
\rho(\lambda\mI-\Lambda)=\begin{pmatrix}
1&0&\cdots&0&-\mu_1\\
0&1&\cdots&0&-\mu_2\\
0&0&\ddots&0&\vdots\\
0&0&\cdots&1&-\mu_{m-1}\\
0&0&\cdots&0&0
\end{pmatrix}
\end{equation}
where $\mu_i\in\CZ^\circ$ for $i=1, \ldots, m-1$.  The normalized eigenvector $\mu\in\CZ^m$ associated with $\lambda$ is then of the form  $\mu=(\mu_1, \ldots, \mu_{m-1}, 1)^\intercal$, where $\mu_i\in\CZ^\circ$ for each $i$.  

\begin{description}
\item [{\bf Claim.}]
For each $k=1, \ldots, m-1$, 
\begin{equation}\label{mu form}
\mu_{m-k}=\sum_{\ell=1}^{k}\sum_{I\subseteq[m-1]}\gamma_{\ell,I\cup\{m\}}\zeta_{I\cup\{m\}}\prod_{t\in I\cap[m-1]}{\ell_t}^{-1},
\end{equation}
where $(-1)^{|I\cup\{m\}|}\gamma_{\ell, I\cup\{m\}}$ is the number of paths and PWICs from $v_{\ell}$ to $v_m$ on vertices indexed by $I\cup\{m\}$.  
\end{description}

Recall that $\lambda\mI-\Lambda$ has upper triangular form \begin{equation}
\tau(\lambda\mI-\Lambda)=\left(\begin{matrix}
\alpha_1&*&\cdots&*&\eta_1\\
0&\alpha_2&\cdots&*&\eta_2\\
0&0&\ddots&\vdots&\vdots\\
\vdots&\vdots&0&\alpha_{m-1}&\eta_{m-1}\\
0&0&\cdots&0&0
\end{matrix}\right),  
 \end{equation}
where $\alpha_\ell\in\CZ^\times$ and $\eta_\ell\in\CZ^\circ$ for $\ell=1, \ldots, m$. 

Proceeding as before using elementary row operations, set $\rho_0=\tau(\lambda\mI-\Lambda)$ and let $\rho_k$ be the matrix obtained after $k$ steps of the (lower) triangularization process for $k=1, \ldots, m-1$.  

As basis for induction, we have \begin{equation}
(\rho_1)_{ij}=\begin{cases}
1&i=j=m-1;\\
(\alpha_{m-1})^{-1}\eta_{m-1}& i=m-1, j=m;\\
\eta_i-(\alpha_{m-1})^{-1}(\tau_i)_{im}\eta_{m-1}& i<m-1, j=m;\\
(\tau_i)_{ij}& i<j<m-1; \\
0&\text{\rm otherwise.}
\end{cases}
\end{equation}
In particular, keeping in mind that $\zeta_{\{m\}}$ annihilates each $\eta_i$, 
\begin{align*}
\mu_{m-1}&=-(\alpha_{m-1})^{-1}\eta_{m-1}\\
&=-(d_{m-1}-d)^{-1}\left(\varepsilon_{m-1,m}-\sum_{I\subseteq[m-2]}\gamma_{I\cup\{m\}}\prod_{t\in I\cap[m-2]}{\ell_t}^{-1}\right)\\
&=\sum_{I\subseteq[m-1]}\gamma_{I\cup\{m\}}\zeta_{I\cup\{m\}}\prod_{t\in I\cap[m-1]}{\ell_t}^{-1}
\end{align*}
where $(-1)^{|I\cup\{m\}|}\gamma_{I\cup\{m\}}$ is the number of paths and PWICs from $v_{m-1}$ to $v_m$ on vertices indexed by $I\cup\{m\}$.   

Assume now that for some $k\ge1$, we have \begin{equation}\label{IH}
\mu_{m-k}=\sum_{I\subseteq[m-1]}\gamma_{I\cup\{m\}}\zeta_{I\cup\{m\}}\prod_{t\in I\cap[m-1]}{\ell_t}^{-1}
\end{equation}
where $(-1)^{|I\cup\{m\}|}\gamma_{I\cup\{m\}}$ is the number of paths and PWICs from $v_{m-1}$ to $v_m$ on vertices indexed by $I\cup\{m\}$.  It follows that 
\begin{align*}
\mu_{m-(k+1)}&=-(d_{m-k-1)}-d)^{-1}\\
&\times\left(\varepsilon_{m-k-1,m}-\sum_{\ell=1}^k
\sum_{I\subseteq[m-(\ell+1)]}\gamma_{I\cup\{m\}}\prod_{t\in I\cap[m-(\ell+1)]}{\ell_t}^{-1}\right)\\
&=\sum_{I\subseteq[m-1]}\gamma_{\ell,I\cup\{m\}}\zeta_{I\cup\{m\}}\prod_{t\in I\cap[m-1]}{\ell_t}^{-1},
\end{align*}
where $(-1)^{|I\cup\{m\}|}\gamma_{\ell, I\cup\{m\}}$ is the number of paths and PWICs from $v_{\ell}$ to $v_m$ on vertices indexed by $I\cup\{m\}$.

Proceeding as before using elementary row operations, the following form is obtained:  
\begin{equation}
\rho_{m-1}=\left(\begin{matrix}
\alpha_1&0&0&\cdots&-\alpha_1\mu_1\\
0&\alpha_2&0&\cdots&-\alpha_2\mu_2\\
0&0&\ddots&0&\vdots\\
\vdots&\vdots&0&\alpha_{m-1}&-\alpha_{m-1}\mu_{m-1}\\
0&0&\cdots&0&0
\end{matrix}\right).  
\end{equation}
Multiplying the nonzeros rows by inverses of the diagonal elements gives the final form seen in \eqref{RRF}.  

It now follows that for each $\ell=1, \ldots, m-1$, the $\ell$th component of $\Lambda \mu$ is given by 
\begin{align*}
\langle \ell\vert \Lambda \mu\rangle=d_\ell\mu_\ell+\sum_{\{j: \{v_j,v_\ell\}\in E\}}\zeta_{\{j\}}\mu_j&=\langle \ell\vert\lambda \mu\rangle\\
&=\lambda \mu_\ell=d\mu_\ell.
\end{align*}
Here, we have used the fact that $\zeta_{\{m\}}\mu_\ell=0=\zeta_{\{m\}}\zdu\lambda$ for $\ell=1, \ldots, m-1$ and $d=\zco\lambda$.  Thus,
\begin{equation*}
\mu_\ell=\frac{1}{d-d_\ell}\sum_{\{j: \{v_j,v_\ell\}\in E\}}\zeta_{\{j\}}\mu_j.
\end{equation*}
Moreover, $\lambda=\langle m\vert\Lambda \mu\rangle$ implies \begin{align*}
\lambda=d+\sum_{\{j: \{v_j,v_m\}\in E\}}\zeta_{\{j\}}\mu_j.
\end{align*}
\end{proof}

Observing that for each $\ell=1, \ldots, m-1$, the $\ell$th component of the eigenvector $\mu$ is given by $\mu_\ell=\displaystyle\sum_{I\ne\varnothing}\langle \mu_\ell, \zeta_I\rangle\zeta_I$,  we see that the $\ell$th component of $\mu$ gives an algebraic representation of all paths and PWICs $v_\ell\to v_m$ in the graph $G$.

\section{Generalizations of the Zeon Combinatorial Laplacian}\label{generalizations}

We now turn to generalizations of the zeon combinatorial Laplacian that enumerate all cycles of a finite graph.  All eigenvalues of these generalized zeon Laplacians can be obtained from the exponential of the graph's nilpotent adjacency matrix.  Hence the characteristic polynomial itself can be computed this way.
 
The first approach is to simply replace vertex degrees with a convenient labeling of the graph's vertices.  

Given a finite simple graph $G$ on $m$ vertices $V=\{v_1, v_2, \ldots, v_m\}$, let $\mathbf{f}=(f_1, \ldots, f_m)\in\mathbb{N}^m$ be an ordered $m$-tuple of distinct positive integers.   Let $\Delta_{\mathbf{f}}$ be the $m\times m$ diagonal matrix whose main diagonal is $\mathbf{f}$.  Letting $\Psi$ be the nilpotent adjacency matrix of $G$, the matrix $\Lambda_{\mathbf{f}}=\Delta_{\mathbf{f}}-\Psi\in\Mat(m,\CZ)$ is referred to as the {\em $\mathbf{f}$-labeled zeon combinatorial Laplacian}. 

\begin{corollary}\label{discrete spectrum}
Let $G=(V,E)$ be a simple finite graph on $m$ vertices, and let $\mathbf{f}=(f_1, \ldots, f_m)\in\mathbb{N}^m$ be an ordered $m$-tuple of distinct positive integers.   Let $\Lambda_{\mathbf{f}}$ be the associated ${\mathbf{f}}$-labeled zeon combinatorial Laplacian of $G$.  Then, for each $\ell=1, \ldots, m$, $\Lambda_{\mathbf{f}}$ has a spectrally simple eigenvalue $\lambda_\ell$ satisfying 
\begin{equation*}
\lambda_\ell = f_\ell + \sum_{I\subseteq [m]}\omega_I\zeta_{I}\prod_{j\in I\setminus\{\ell\}}(f_\ell-f_j)^{-1},
\end{equation*}
where $(-1)^{|I|}\omega_I$ is the number of cycles based at $v_\ell$ on vertices indexed by $I$.
\end{corollary}

Given a graph $G$ and generalized zeon Laplacian $\Lambda_{\mathbf{f}}$ as described in Corollary \ref{discrete spectrum}, each vertex $v$ corresponds to a unique eigenvalue $\lambda_v$ with zeon eigenvector (eigenstate) $\varepsilon_v$.

\subsection{Symmetric Zeon Laplacians and a Quantum Probabilistic Interpretations of Zeon Matrices}

Recalling that self-adjoint linear operators on Hilbert spaces can be interpreted as quantum observables, we seek an analogous interpretation of the zeon combinatorial Laplacian.  

Two vectors $v_1,v_2\in{\CZ}^m$ are said to be {\em orthogonal} if and only if $\langle v_1\vert v_2\rangle=0$.   Note that any collection of zeon vectors $\{v_1, \ldots, v_m\}$ spanning ${\CZ}^m$ can be orthogonalized via Gaussian elimination.  Normalizing then yields an orthonormal basis for ${\CZ}^m$.  

Given a normalized element $v\in{\CZ}^m$, the matrix $vv^\dag$ represents orthogonal projection onto $\spn(\{v\})$.   

A matrix $A\in\Mat(m,\CZ)$ satisfying $A^\dag=A$ is clearly self-adjoint with respect to the inner product \eqref{inner prod def}.  Eigenvalues of self-adjoint zeon matrices are elements of the real zeon algebra $\Z_n$; i.e., $\chi_A(\lambda)=0$ implies $\lambda=\overline{\lambda}$ when $A$ is self-adjoint.   Further, eigenvectors of self-adjoint zeon matrices associated with distinct eigenvalues are orthogonal.  Given distinct eigenvalues $\lambda_1$ and $\lambda_2$ associated with eigenvectors $v_1$ and $v_2$, respectively, 
\begin{align*}
\lambda_1\langle v_1\vert v_2\rangle&=\langle Av_1\vert v_2\rangle\\
&=\langle v_1\vert A v_2\rangle=\lambda_2 \langle v_1\vert v_2\rangle,
\end{align*}
implying $\langle v_1\vert v_2\rangle=0$ since $\zco \lambda_1\ne \zco\lambda_2$.

\begin{description}
\item [{\bf Zeon Spectral Theorem.}]
Let $A\in\Mat(m,\CZ)$ be a self-adjoint zeon matrix with $m$ spectrally simple eigenvalues.  Let $v_1, \ldots, v_m$ denote normalized zeon eigenvectors associated with these eigenvalues and set $\pi_j=v_j{v_j}^\dag$ for $j=1, \ldots, m$.  Then, \begin{equation*}
A=\bigoplus_{j=1}^m \lambda_j\pi_j.
\end{equation*}
\end{description}

Given finite simple graph $G$ on $m$ vertices, let $\Delta_*$ denote any $m\times m$ invertible diagonal matrix with distinct diagonal entries $d_1, \ldots, d_m\in\mathbb{Q}^+$, and let $\Psi_*=\Psi+\Psi^\dag$, where $\Psi$ is the nilpotent adjacency matrix of $G$.  Then $\Delta_*-\Psi_*$ is symmetric and diagonalizable and has spectral decomposition $\Delta_*-\Psi_*=\displaystyle\bigoplus_{j=1}^n\lambda_j\pi_j$ for eigenvalues $\lambda_j$ and associated eigenspace projections $\pi_j$.

Letting $\Lambda_*=\Delta_*-\Psi_*$, it follows immediately that $\Lambda_*$ is symmetric and therefore self-adjoint with respect to the inner product on ${\CZ}^m$ defined by 
\begin{equation*}
\langle x\vert y\rangle =y^\dag x.
\end{equation*}

Note that $\Psi\Psi^\dag=\zm$.  As an immediate consequence, \begin{equation*}
(\Psi+\Psi^\dag)^k=\sum_{j=0}^k {\Psi^\dag}^j\Psi^{k-j}
\end{equation*}
for each positive integer $k$.  Moreover, $\Psi^\ell=\zm\Leftrightarrow {\Psi^\dag}^\ell=\zm$, so that $\kappa(\Psi)=\kappa(\Psi^\dag)$.

Observing that $\langle v_i\vert \Psi\vert v_j\rangle=\langle v_j\vert \Psi^\dag\vert v_i\rangle$, the next result follows immediately from Theorem \ref{nil-structure theorem}. 

\begin{lemma} [Nil-Structure Theorem for $\Psi^\dag$]
Let $\Psi$ be the nilpotent adjacency matrix of an
$m$-vertex graph $G$, and let $k\in\mathbb{N}$.  For $1\le i\ne j \le m$, 
\begin{equation*}
\left<v_i\vert{\Psi^\dag}^k\vert \zeta_{\{j\}}\right>=
\sum_{\genfrac{}{}{0pt}{}{I\subseteq V}{|I|=k+1}}\omega_I \zeta_{I},\end{equation*}
where $\omega_I$ denotes the number of $k$-paths from $v_j$ to $v_i$ on vertices indexed by $I$.  Further,  when $1\le i \le m$,
\begin{equation*}
\left<v_i\vert{\Psi^\dag}^k\vert v_i\right>=\sum_{\genfrac{}{}{0pt}{}{I\subseteq V}{|I|=k}}\omega_I \zeta_{I},\end{equation*}
where $\omega_I$ denotes the number of $k$-cycles on vertex set $I$ based at $v_i\in I$.  
\end{lemma}

Having an understanding of both $\Psi$ and $\Psi^\dag$, attention is now given to the symmetric nilpotent adjacency matrix $\Psi+\Psi^\dag$.

\begin{theorem}[Symmetric Nil-Structure Theorem]
Let $\Psi$ be the nilpotent adjacency matrix of an
$m$-vertex graph $G$, and let $k\in\mathbb{N}$.  For $1\le i\ne j \le m$, 
\begin{equation*}
\left<\zeta_{\{i\}}\vert(\Psi+\Psi^\dag)^k\vert j\right>=
\sum_{\genfrac{}{}{0pt}{}{I\subseteq V}{|I|=k+1}}\omega_I \zeta_{I},\end{equation*}
where $\omega_I$ denotes the number of $k$-paths from $v_i$ to $v_j$ on vertices indexed by $I$.  Further,  when $1\le i \le m$,
\begin{equation*}
\left<i\vert(\Psi+\Psi^\dag)^k\vert i\right>=2\sum_{\genfrac{}{}{0pt}{}{I\subseteq V}{|I|=k}}\omega_I \zeta_{I},\end{equation*}
where $\omega_I$ denotes the number of $k$-cycles on vertex set $I$ based at $v_i\in I$.  
\end{theorem}

\begin{proof}
Observing that $\Psi\Psi^\dag=\zm$, we see that for any positive integer $k$, \begin{equation*}
(\Psi+\Psi^\dag)^k=\sum_{j=0}^k {\Psi^\dag}^j\Psi^{k-j}.
\end{equation*}
By construction, all nonzero entries in the $i$th row of $\Psi^\dag$ are $\zeta_{\{i\}}$, so that $\langle\zeta_{\{i\}}\vert \Psi^\dag=\zm$ for each $i$.  Thus, when $i\ne j$, 
\begin{align*}
\left<\zeta_{\{i\}}\vert(\Psi+\Psi^\dag)^k\vert j\right>&=\sum_{\ell=0}^k\left<\zeta_{\{i\}}\vert{\Psi^\dag}^\ell\Psi^{k-\ell}\vert j\right>\\
&=\left<\zeta_{\{i\}}\vert\Psi^k\vert j\right>\\
&=\sum_{\genfrac{}{}{0pt}{}{I\subseteq V}{|I|=k+1}}\omega_I \zeta_{I},\end{align*}
where $\omega_I$ denotes the number of $k$-paths from $v_i$ to $v_j$ on vertices indexed by $I$.

Write $\Psi=(\psi_{ij})$ and $\Psi^\dag=({\psi^\dag}_{ij})$, where ${\psi^\dag}_{ij}=\psi_{ji}$.   When $1\le j\le m-1$, it follows that for each $i=1, \ldots, m$,  
\begin{align*}
\langle i\vert {\Psi^\dag}^j\Psi^{k-j}\vert i\rangle&=\sum_{\ell=1}^k {{\psi^\dag}^j}_{i\ell}{\psi^{k-j}}_{\ell i}\\
&=\sum_{\ell=1}^m {\psi^j}_{\ell i}{\psi^{k-j}}_{\ell i}\\
&=0
\end{align*}
because $\zeta_{\{i\}}$ annihilates both ${\psi^j}_{\ell i}$ and ${\psi^{k-j}}_{\ell i}$.  Hence, the only nonzero terms along the main diagonal of $(\Psi+\Psi^\dag)^k$ occur in the cases $j=0$ and $j=k$.  Accordingly, we have \begin{align*}
\langle i\vert(\Psi+\Psi^\dag)^k\vert i\rangle&=\sum_{j=0}^k \langle i\vert {\Psi^\dag}^j\Psi^{k-j}\vert i\rangle\\
&=\langle i\vert \Psi^{k}\vert i\rangle + \langle i\vert {\Psi^\dag}^k\vert i\rangle\\
&=2\sum_{|I|=k}\omega_I\zeta_I,
\end{align*}
where $\omega_I$ denotes the number of $k$-cycles based at $v_i$ on vertices indexed by $I$.
\end{proof}

Observing that $\Lambda^\dag=\Delta-\Psi^\dag$, the eigenvectors of $\Lambda^\dag$ are characterized in similar fashion to those of $\Lambda$.  Assuming $\lambda$ is an eigenvalue of $\Lambda^\dag$ associated with vertex $v_m$ and following the approach of Theorem \ref{Laplacian eigenvector}, $\Lambda^\dag$ has an eigenvector of the form $\nu=(\nu_1, \ldots,  \nu_{m-1}, 1)^\intercal$ where $\nu_i\in{\CZ_m}^\circ$ for $i=1, \ldots, m-1$.  It follows that for each $\ell=1, \ldots, m-1$\begin{equation}\label{eigenvector characterization 2}
\langle \nu_\ell, \zeta_I\rangle=\gamma_I(d-d_\ell)^{-1}\prod_{j\in I}(d-d_j)^{-1},
\end{equation}
where $(-1)^{|I|}\gamma_I$ denotes the number of paths and PWICs $v_m\to v_\ell$ on vertices indexed by $I$.  Hence,  $\nu_\ell=\displaystyle\sum_{I\ne\varnothing}\langle \nu_\ell, \zeta_I\rangle\zeta_I$, providing an algebraic representation of all paths and PWICs $v_m\to v_\ell$ in the graph $G$.

Given the zeon combinatorial Laplacian $\Lambda=\Delta-\Psi$ of a simple graph $G$, we define the {\em symmetric zeon combinatorial Laplacian of $G$} to be the matrix $\Lambda_{\rm sym}=\Delta-(\Psi+\Psi^\dag)$.  When $G$ has a vertex $v$ of unique degree,  $\Lambda_{\rm sym}$ has a zeon eigenvalue $\lambda_v$ that can be characterized in terms of cycles based at $v$ as follows. 

\begin{theorem}\label{Lambda_sym eigenvalue}
Let $\Lambda=\Delta-\Psi$ be the zeon combinatorial Laplacian of a simple graph on $m$ vertices, and let $\Lambda_{\rm sym}=\Delta-(\Psi+\Psi^\dag)$.  If $\lambda$ is the zeon eigenvalue of $\Lambda$ associated with a vertex of unique degree $d=\zco\lambda$, then $\lambda_*=\zco\lambda+2\zdu \lambda$ is the unique zeon eigenvalue of $\Lambda_{\rm sym}$ satisfying $\zco\lambda_*=\zco\lambda$.  
\end{theorem}

\begin{proof}
As in the proof of Theorem \ref{laplacian eigenvalue}, let us consider the reduced row echelon form of $\lambda_*\mI-\Lambda_{\rm sym}$, assuming the last vertex, $v_m$, is associated with $\lambda_*$.  Writing $\zeta_{i\oplus j}=\zeta_{\{i\}}+\zeta_{\{j\}}$ for ease of notation and setting $\ell_i=(d-d_i)+\zdu \lambda_s$, we begin with \begin{equation}
 \lambda_*\mI-\Lambda_{\rm sym}=\left(\begin{matrix}
\ell_1&\varepsilon_{12}\zeta_{1\oplus 2}&\cdots&\cdots&\varepsilon_{1 m}\zeta_{1\oplus m}\\
\varepsilon_{21}\zeta_{2\oplus 1}&\ell_2&\varepsilon_{23}\zeta_{2\oplus 3}&\cdots&\varepsilon_{2m}\zeta_{2\oplus m}\\
\varepsilon_{31}\zeta_{3\oplus 1}&\varepsilon_{32}\zeta_{3\oplus 2}&\ell_3&\cdots&\varepsilon_{3m}\zeta_{3\oplus m}\\
\vdots&\vdots&\vdots&\ddots&\vdots\\
\varepsilon_{m1}\zeta_{m\oplus 1}&\cdots&\cdots&\cdots&\zdu \lambda_*
\end{matrix}\right).  
 \end{equation}
As before,  $\varepsilon_{ij}$ is defined by 
\begin{equation*}
\varepsilon_{i j}=\begin{cases}
1&\text{\rm if }\{v_i,v_j\}\in E,\\
0&\text{\rm otherwise.}
\end{cases}
\end{equation*}

Once again, the only invertible elements lie along the main diagonal, and the matrix can be put in the following form without changing the determinant:
\begin{equation}\label{mostly triangular form of symmetric laplacian}
\phi(\lambda_*\mI-\Lambda_{\rm sym})=\left(\begin{matrix}
\alpha_1&*&\cdots&*&\eta_1\\
0&\alpha_2&\cdots&*&\eta_2\\
0&0&\ddots&\vdots&\vdots\\
\vdots&\vdots&0&\alpha_{m-1}&\eta_{m-1}\\
\varepsilon_{m1}\zeta_{m\oplus 1}&\cdots&\cdots&\cdots&\zdu \lambda_*
\end{matrix}\right),  
 \end{equation}
where $\alpha_\ell\in\CZ^\times$ for $\ell=1, \ldots, m-1$, and $\eta_\ell\in\CZ^\circ$ for $\ell=1, \ldots, m$.   

In the last step of the triangularization, we obtain \begin{equation}\label{triangular form of symmetric laplacian}
\tau(\lambda_*\mI-\Lambda_{\rm sym})=\left(\begin{matrix}
\alpha_1&*&\cdots&*&\eta_1\\
0&\alpha_2&\cdots&*&\eta_2\\
0&0&\ddots&\vdots&\vdots\\
\vdots&\vdots&0&\alpha_{m-1}&\eta_{m-1}\\
0&\cdots&\cdots&0&\eta_m
\end{matrix}\right),  
 \end{equation}
where $\eta_m=0$ since $0=|\lambda_*\mI-\Lambda_{\rm sym}|=\eta_m\prod_{j=1}^{m-1}\alpha_j$. 

Notably, \begin{equation}\label{last eta}
 \eta_m=\zdu\lambda_*-\sum_{\{j: \{v_j,v_m\}\in E\}}\zeta_{m\oplus j}\eta_j,
 \end{equation}
so that our work will be finished by showing $\sum_{\{j: \{v_j,v_m\}\in E\}}\zeta_{m\oplus j}\eta_j = 2\zdu\lambda$.

\begin{description}
\item [{\bf Claim.}] For $j=1, \ldots, m-1$, 
\begin{equation}\label{eta claim}
\eta_j=\varepsilon_{jm}\zeta_{j\oplus m}-2\sum_{I\subseteq [j-1]}\gamma_{I\cup\{m\}}\zeta_{I\cup\{m\}}\prod_{t\in I\cap[j-1]}{\ell_t}^{-1},
\end{equation}
where $(-1)^{|I|\cup\{m\}}\gamma_{I\cup\{m\}}$ is the number of paths and PWICs from $v_j$ to $v_m$ on vertices indexed by $I\cup\{m\}$. 
\\

\noindent{\em Proof of claim.}
Any nonzero off-diagonal entry of $\Lambda_{\rm sym}$ is an element of the form
\begin{equation*}
\langle i\vert\Lambda_{\rm sym}\vert j\rangle = \zeta_{i\oplus j}.
\end{equation*}
In light of the symmetric nil-structure theorem, such an entry can be interpreted as the sum of two one-step paths in the graph: $\zeta_{\{i\}}$ represents a one-step path from $v_j\to v_i$, while $\zeta_{\{j\}}$ represents a one-step path from $v_i\to v_j$.  Since the graph is undirected, walks $v_i\to v_j$ are in one-to-one correspondence with walks $v_j\to v_i$.  The rest follows {\em mutatis mutandis} from Cases 3 and 4 of Theorem \ref{laplacian eigenvalue}.  This concludes the proof of the claim.
\end{description}

Observing that $\zeta_{i\oplus m}\zeta_{m\oplus i}=2\zeta_{\{i,m\}}$ and substituting \eqref{eta claim} into \eqref{last eta}, the following now becomes apparent:
\begin{align*}
 \eta_m&=\zdu\lambda_*-\sum_{\{j: \{v_j,v_m\}\in E\}}\zeta_{m\oplus j}\eta_j\\
 &=\zdu\lambda_*-\sum_{\{j: \{v_j,v_m\}\in E\}}\zeta_{m\oplus j}\zeta_{j\oplus m}-2\sum_{I\subseteq [j-1]}\left[\gamma_{I\cup\{m\}}\zeta_{I\cup\{m\}}\prod_{t\in I\cap[j-1]}{\ell_t}^{-1}\right]\\
 &=\zdu\lambda_*-2\sum_{\{j: \{v_j,v_m\}\in E\}}\left[\zeta_{\{j,m\}}-\sum_{I\subseteq [j-1]}\gamma_{I\cup\{m\}}\zeta_{I\cup\{m\}}\prod_{t\in I\cap[j-1]}{\ell_t}^{-1}\right]\\
 &=\zdu\lambda_*-2\sum_{I\subseteq[m]}\omega_I\zeta_I\prod_{j\in I\setminus\{m\}}(d-d_j)^{-1}\\
 &=\zdu\lambda_*-2\zdu \lambda.
 \end{align*}
Since $\eta_m=0$, the proof is complete.
\end{proof}

The eigenvectors of $\Lambda_{\rm sym}$ are characterized below. 

\subsection*{Eigenvectors of $\Lambda_{\rm sym}$}

Let $\lambda_*$ denote the eigenvalue of the symmetric zeon Laplacian $\Lambda_{\rm sym}$ associated with vertex $v_m$.  Let $\xi$ be the unit eigenvector satisfying $\Lambda_{\rm sym} \xi=\lambda_*\xi$ and $\langle \xi\vert m\rangle=1$.

\begin{description}
\item[{\bf Claim.}] For each $j=1, \ldots, m$, the following equality holds:
\begin{equation*}
\langle \xi\vert \zeta_{\{j\}}\rangle = \langle \zeta_{\{m\}}\xi\vert j\rangle.
\end{equation*}
Equivalently, 
\begin{equation*}
\langle (\zeta_{\{m\}}-\zeta_{\{j\}})\xi\vert j\rangle = 0
\end{equation*}
for each $j=1, \ldots, m$.  Hence, $\zeta_{\{m\}}-\zeta_{\{j\}}\in\Ann \langle\xi\vert j\rangle$ for each $j$.

\end{description}

Letting $d=\zco\lambda_*$ and writing $\xi=(\xi_1, \xi_2,\ldots, \xi_{m-1},1)^\intercal$, we now see that
\begin{align*}
\lambda_*&=d+\langle \Lambda_{\rm sym}\xi\vert m\rangle\\
&=d+\sum_{\{j: \{v_j,v_m\}\in E\}}(\zeta_{\{m\}}+\zeta_{\{j\}})\xi_j\\
&=d+2\zeta_{\{m\}}\sum_{\{j: \{v_j,v_m\}\in E\}}\xi_j\\
&=d+2\sum_{\{j: \{v_j,v_m\}\in E\}}\zeta_{\{j\}}\xi_j.
\end{align*}
Letting $\nu$ denote the unit eigenvector of $\Lambda$ satisfying $\Lambda \nu=\lambda\nu$ and $\langle \nu\vert m\rangle=1$,  $\zdu\lambda_*=2\zdu\lambda$ now implies 
\begin{align*}
\zdu\lambda_*&=2\sum_{\{j:\{v_j,v_m\}\in E\}}\zeta_{\{j\}}\xi_j\\
&=2\sum_{\{j:\{v_j,v_m\}\in E\}}\zeta_{\{j\}}\nu_j.
\end{align*}  

\begin{proposition}[Eigenvectors of the Symmetric Zeon Laplacian]\label{symmetric eigenvectors}
Let $\Lambda$ denote the zeon combinatorial Laplacian of a graph $G=(V,E)$ on $m$ vertices $\{v_1, \ldots, v_m\}$, such that $v_k$ is of unique degree (i.e., label) $d$.  Let $\lambda=d+\zdu\lambda$ be the zeon eigenvalue of $\Lambda$ associated with vertex $v_k$, and let $\nu\in\CZ^m$ denote the associated unit eigenvector whose $k$th component is 1.  Then, the symmetric zeon Laplacian $\Lambda_{\rm sym}$ has eigenvalue $\lambda_*=d+2\zdu\lambda$ with associated unit eigenvector $\xi\in\CZ^m$ satisfying \begin{equation}
\zeta_{\{k\}}\xi=\Gamma\nu,
\end{equation}
where $\Gamma={\rm diag}(\zeta_{\{1\}},\zeta_{\{2\}},\ldots, \zeta_{\{m\}})$.
\end{proposition}

\begin{proof}
Observing that 
\begin{equation*}
\Gamma\nu=\displaystyle\begin{pmatrix}
\zeta_{\{1\}}&0&\cdots&0\\
0&\ddots&0&0\\
\vdots&\cdots&\zeta_{\{m-1\}}&0\\
0&\cdots&0&\zeta_{\{m\}}\end{pmatrix}\begin{pmatrix}
\nu_1\\
\vdots\\
\nu_{m-1}\\
\nu_m
\end{pmatrix}=\begin{pmatrix}
\zeta_{\{1\}}\nu_1\\
\vdots\\
\zeta_{\{m-1\}}\nu_{m-1}\\
\zeta_{\{m\}}\nu_m
\end{pmatrix},
\end{equation*}
the result is established by showing that $\zeta_{\{k\}}\xi_\ell=\zeta_{\{\ell\}}\nu_\ell$ for each $\ell=1,\ldots, m$.  Since $\nu_k=\xi_k=1$, the result is trivial when $\ell=k$. Assuming $\ell\ne k$, we see that the $\ell$th entry of $\xi$ represents sums of PWICs $v_k\to v_\ell$ and $v_\ell\to v_k$ by Theorems \ref{Laplacian eigenvector} and \ref{Lambda_sym eigenvalue}.  Multiplication by $\zeta_{\{k\}}$ eliminates both paths and PWICs $v_\ell\to v_k$ and eliminates PWICs $v_k\to v_\ell$, leaving only paths.  Each of these paths can also be viewed as a path $v_\ell\to v_k$ because the graph is undirected.  Similarly, $\nu_\ell$ represents paths and PWICs $v_\ell\to v_k$, while multiplication by $\zeta_{\{\ell\}}\nu_\ell$ eliminates those terms representing PWICs $v_\ell\to v_k$.
\end{proof}

\begin{figure}
$$\fbox{\includegraphics[width=160pt]{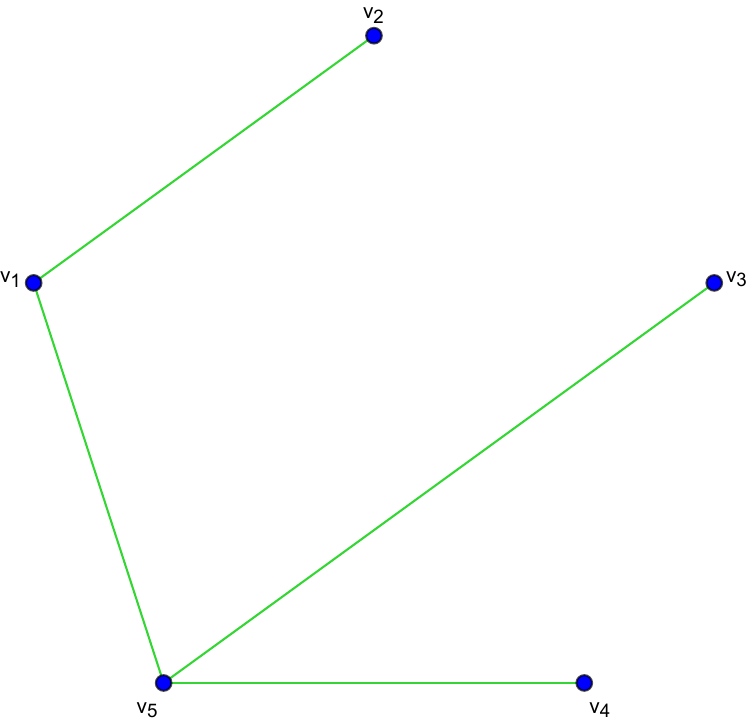}\hskip10pt\includegraphics[width=150pt]{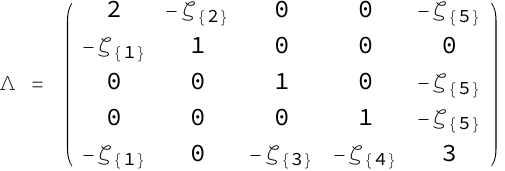}}$$
\caption{A 5-vertex graph and its zeon combinatorial Laplacian.\label{5Vgraph}}
\end{figure}
\begin{example}
Consider the simple graph of Fig. \ref{5Vgraph}.  The eigenvalue and associated zeon eigenvector of $\Lambda$ corresponding to vertex $v_5$ are \begin{equation*}
\lambda=3+\zeta_{\{1,5\}}+\frac{1}{2}\zeta_{\{3,5\}}+\frac{1}{2}\zeta_{\{4,5\}},\hspace{20pt}\nu=\begin{pmatrix}-\zeta_{\{5\}}-\frac{1}{2} \zeta _{\{1,2,5\}}\\[.1em]
\frac{1}{2} \zeta _{\{1,5\}}\\[.1em]
-\frac{\zeta _{\{5\}}}{2}\\[.1em]
-\frac{\zeta _{\{5\}}}{2}\\[.1em]
1\end{pmatrix}.
\end{equation*}
The eigenvalue and associated zeon eigenvector of $\Lambda_{\rm sym}$ corresponding to vertex $v_5$ are $\lambda_*=3+2\zeta_{\{1,5\}}+\zeta_{\{3,5\}}+\zeta_{\{4,5\}}$ and \begin{equation*}
\xi=\begin{pmatrix}-\zeta _{\{1,2,5\}}+\zeta _{\{1,3,5\}}+\zeta _{\{1,4,5\}}-\zeta _{\{1\}}-\zeta _{\{5\}}\\[.1em]
\frac{1}{2} \zeta _{\{1,2\}}+\frac{1}{2} \zeta _{\{1,5\}}+\frac{1}{2} \zeta _{\{2,5\}}-\frac{3}{4} \zeta _{\{1,2,3,5\}}-\frac{3}{4} \zeta _{\{1,2,4,5\}}\\[.1em]
\frac{1}{2} \zeta _{\{1,3,5\}}+\frac{1}{4} \zeta _{\{3,4,5\}}-\frac{\zeta _{\{3\}}}{2}-\frac{\zeta _{\{5\}}}{2}\\[.1em]
\frac{1}{2} \zeta _{\{1,4,5\}}+\frac{1}{4} \zeta _{\{3,4,5\}}-\frac{\zeta _{\{4\}}}{2}-\frac{\zeta _{\{5\}}}{2}\\[.1em]
1\end{pmatrix},
\end{equation*}
respectively.  Straightforward computation shows that \begin{equation*}
\zeta_{\{5\}}\xi=\begin{pmatrix}
-\zeta _{\{1,5\}}\\
\frac{1}{2} \zeta _{\{1,2,5\}}\\[.1em]
-\frac{1}{2} \zeta _{\{3,5\}}\\[.1em]
-\frac{1}{2} \zeta _{\{4,5\}}\\
\zeta _{\{5\}}\end{pmatrix}=\begin{pmatrix}\zeta_{\{1\}}&0&0&0&0\\
0&\zeta_{\{2\}}&0&0&0\\
0&0&\zeta_{\{3\}}&0&0\\
0&0&0&\zeta_{\{4\}}&0\\
0&0&0&0&\zeta_{\{5\}}
\end{pmatrix}\nu.
\end{equation*}
\end{example}

\subsubsection{The $\q$-Zeon Combinatorial Laplacian}

Finally, we turn to a self-adjoint $\CZ$-linear operator whose eigenvalues characterize all cycles in a simple graph.  

\begin{definition}[$\q$-Zeon Combinatorial Laplacian]\label{quantum laplacian def}
Let $G=(V,E)$ be a simple graph on $m$ vertices $E=\{v_1, \ldots, v_m\}$, and let $\q=(q_1, \ldots, q_m)\in{\mathbb{Q}^+}^m$ satisfy $i\ne j\Rightarrow q_i\ne q_j$.  We define the {\em zeon $\q$-combinatorial Laplacian of $G$} to be the matrix $\Lambda_\q=\Delta_\q-(\Psi+\Psi^\dag)$, where $\Delta_\q={\rm diag}(\q)$.  
\end{definition}

The matrix $\Lambda_\q$ can now be considered a quantum observable.  The values taken by the observable are the zeon eigenvalues $\lambda_{1}, \ldots, \lambda_{m}$, each associated with a unique vertex of the graph.   The pure states of $\Lambda_\q$ are the normalized eigenvectors $\xi_j\in\CZ^m$ associated with the eigenvalues $\lambda_{j}$.

Letting $\xi_j\in\CZ^m$ be a normalized eigenvector associated with the vertex $v_j\in V$, the {\em expectation of $\Lambda_\q$ in the state $\xi_j$} is given by 
\begin{equation*}
\langle \xi_j\vert \Lambda_\q \vert \xi_j\rangle= q_j+ 2\sum_{I\subseteq [m]}\omega_I\zeta_{I}\prod_{\ell\in I\setminus\{j\}}(q_j-q_\ell)^{-1},  
\end{equation*}
where $(-1)^{|I|}\omega_I$ is the number of cycles based at $v_j$ on vertices indexed by $I$.

As seen in Proposition \ref{symmetric eigenvectors}, the eigenvalues and associated eigenvectors of the symmetric zeon combinatorial Laplacian can be derived from the corresponding values associated with the underlying zeon Laplacian, $\Lambda$.  These values, in turn, can be obtained by examining powers of the underlying nilpotent adjacency matrix.

\section{Concluding Remarks}\label{conclusion}

Beginning with a simple graph $G=(V,E)$ on $m$ vertices, we define the zeon combinatorial Laplacian $\Lambda=\Delta-\Psi$ associated with $G$ in terms of the nilpotent adjacency matrix $\Psi$ of $G$.  In this paper, we have established the following results.

\begin{enumerate}
\item When $G$ has a vertex $v$ of unique degree $d$, the zeon Laplacian $\Lambda$ of $G$ has a unique zeon eigenvalue $\lambda$ whose scalar part is $d$ and whose dual part enumerates the cycles based at $v$.
\item The eigenvector associated with an eigenvalue of $\Lambda$ can be characterized in terms of paths and PWICs terminating at that vertex.
\item Cycles, paths, and PWICs comprising the eigenvalues and eigenvectors can be computed using powers (or exponentials) of the graph's nilpotent adjacency matrix.
\item Eigenvalues and eigenvectors of the symmetric zeon combinatorial Laplacian have been characterized in terms of eigenvalues and eigenvectors of the general zeon combinatorial Laplacian.
\item The zeon $\q$-combinatorial Laplacian can be viewed as a quantum observable whose values enumerate cycles contained in the associated graph. 
\end{enumerate}
\vskip10pt

\subsection{Avenues for Further Exploration}

With these results in hand, a number of ideas come to mind for further research.   Some are included below.  

\begin{enumerate}

\item Graph processes.\\
By considering sequences of zeon combinatorial Laplacians, limit theorems might be established for sequences of graphs evolving over time.
 
\item Quantum algorithms.\\
Since symmetric zeon Laplacians can be regarded as quantum random variables, one is naturally led to consider their potential applications in developing quantum algorithms for solving graph-theoretical problems.

\item Graph colorings and/or enumeration problems in colored graphs.\\
Paths and cycles in colored graphs have been studied in a number of works over the years  \cite{albert,broersma05,frieze,wong}.  Zeon-algebraic formulations of graph coloring problems were first developed in \cite{StaplesStellhorn2017}.  Spectral properties of zeon matrices used in graph coloring problems are natural topics for study.

\end{enumerate}

\section*{Declarations}

\subsection*{Funding}

No funding was received to assist with the preparation of this manuscript.

\subsection*{Conflicts of Interest/Competing Interests}

The author has no relevant financial or non-financial interests to disclose.

\end{document}